\documentclass{amsart}
\usepackage{amssymb}

\usepackage{graphicx}

\def\cal{\mathcal}

\input{tcilatex}
\input tcilatex

\newcommand{\EE}{{\mathbb E}}
\newcommand{\ZZ}{{\mathbb Z}}
\newcommand{\NN}{{\mathbb N}}
\newcommand{\PP}{{\mathbb P}}
\newcommand{\BB}{{\mathbb B}}

\newcommand{\RR}{{\mathbb R}}

\newcommand{\B}{{\cal B}}

\newcommand{\E}{{\cal E}}
\newcommand{\F}{{\cal F}}

\newcommand{\I}{{\cal A}}
\newcommand{\aI}{{\cal I}}

\newcommand{\wQ}{{\widetilde Q}}
\newcommand{\wP}{{\widetilde P}}

\newcommand{\R}{{\cal R}}

\newcommand{\T}{{\cal T}}
\newcommand{\X}{{\cal X}}
\newcommand{\1}{{\bf 1}}
\newcommand{\wL}{{\widetilde {\1}}}

\newcommand{\aP}{{\cal P}}
\newcommand{\bP}{{\mathbb S}}

\newcommand{\wI}{\widetilde{\I}}
\newcommand{\wII}{\widehat{I}}
\newcommand{\wf}{\widetilde{f}}
\newcommand{\wa}{\widetilde{a}}
\newcommand{\wb}{\widetilde{b}}

\newcommand{\wZ}{\widetilde{Z}}

\newcommand{\wpreceq}{\widetilde{\preceq}}

\newcommand{\wH}{\widetilde H}
\newcommand{\wah}{{\widetilde h}}
\newcommand{\lP}{\widehat {\cal P}}
\newcommand{\lS}{\widehat{\cal D}}

\begin{document}
\title[M\"{o}bius Duality]{On M\"{o}bius duality and Coarse-Graining}
\author{Thierry Huillet$^{1}$, Servet Mart\'inez$^{2}$}
\address{$^{1}$Laboratoire de Physique Th\'{e}orique et Mod\*{e}lisation \\
CNRS-UMR 8089 et Universit\'{e} de Cergy-Pontoise, 2 Avenue Adolphe Chauvin,
95302, Cergy-Pontoise, FRANCE\\
$^{2}$ Departamento de Ingenier\'{i}a Matem\*{a}tica \\
Centro Modelamiento Matem\'{a}tico\\
UMI 2807, UCHILE-CNRS \\
Casilla 170-3 Correo 3, Santiago, CHILE.\\
E-mail: Thierry.Huillet@u-cergy.fr and smartine@dim.uchile.cl}

\begin{abstract}
We study duality relations for zeta and M\"{o}bius matrices
and monotone conditions on the kernels. We focus on the cases
of family of sets and partitions.
The conditions for positivity of the dual kernels are stated
in terms of the positive M\"{o}bius cone of functions, which
is described in terms of Sylvester formulae.
We study duality under coarse-graining and show that 
an $h-$transform is needed to preserve stochasticity. 
We give conditions in order that zeta and M\"{o}bius 
matrices admit coarse-graining, and we prove 
they are satisfied for sets and partitions. 
This is a source of relevant examples in genetics on 
the haploid and multi-allelic Cannings models.
\end{abstract}

\maketitle

\textbf{Running title}: M\"{o}bius Duality.

\smallskip

\textbf{Keywords}: \textit{Duality, M\"{o}bius matrices, coarse-graining, 
partitions, Sylvester formula, coalescence.}

\smallskip

\textbf{MSC 2000 Mathematics Subject Classification}: $60$J$10$, $60$J$70$, 
$92$D$25$.

\section{Introduction}
\label{Sec0}

\medskip

We study zeta and M\"{o}bius duality for a finite partially 
ordered space $(\I,\preceq)$ with special
emphasis when this space is a family of sets or of partitions.
We will supply conditions in order that the dual of a
nonnegative kernel $P$ defines a nonnegative kernel $Q$, and  
study relations between these two kernels.

\medskip

Section \ref{SecZM} is devoted to introducing zeta and M\"{o}bius matrices,
as done in \cite{BG, GR}. 
We supply the product formula for the product order which serves to list 
several examples in a unified way.

\medskip

In Section \ref{Sec2} we study zeta and  M\"{o}bius duality relations.
The conditions for positivity preserving are put in terms of the
positive M\"{o}bius cone of functions, which is the class of  positive 
functions having positive image under the M\"{o}bius matrix
(they are called M\"{o}bius monotone in \cite{LS}).
A well known duality relation of this type is 
the Siegmund duality for a finite interval of integers endowed with
the usual order, see \cite{Sieg}. In the general case we can retrieve 
only few of the properties of the Siegmund duality (for its properties 
see \cite{AD, DF, HM}), some of them 
only require that duality preserves positivity, other require stronger 
conditions and we always put them in terms of the positive M\"{o}bius cone.

\medskip

In Section \ref{Sec3} we study  
Sylvester formulae for sets (the well-known inclusion-exclusion 
relations) and for partitions.
To the best of our knowledge, the Sylvester formulae to be found  
in Section \ref{SecSylv2} for partitions, are new. 
These formulae aim at describing the positive M\"{o}bius cone 
and so, in principle, they can give some insight into the problem of when 
duality preserves positivity. 

\medskip

A natural question encountered in the context of zeta and
M\"{o}bius duality is when a duality relation is preserved by 
coarse-graining, that is when we can state some type of duality for 
coarser observations of the processes. Thus, instead of a set
it can be observed the number of elements it contains,
and instead of a partition it can be only access to the size of its atoms.
Coarse-graining duality is studied in Section \ref{Sec4},
the main result being Theorem \ref{thm1} where it is proven that
when the coarse-graining is satisfied, it is required an $h-$transform
in the dual kernel in order that stochasticity is preserved. 
In this section we also show that the conditions for coarse-graining are 
fulfilled for zeta and  M\"{o}bius matrices on sets and partitions. 

\medskip

Finally Section \ref{Sec5} is devoted to some examples of these duality 
relations. In these examples we revisit the haploid Cannings model
and the multi-allelic model with constant population size 
(see \cite{Cn1, Cn2, MM1, MM2}). 
In \cite{MM1,MM2} an ancestor type process was associated to these
models, and their duality was stated.
We will give a set version of these models, showing
they are in duality via a transpose zeta matrix
and that coarse-graining duality modified by an $h-$transform appears in 
a natural way giving the hypergeometric matrix.

\medskip

We point out that many of the concepts we will introduce and even some
of the results we will obtain, are straightforwardly defined or satisfied
in a countable infinite setting. But we prefer to keep a finite
framework for clarity and to avoid technicalities that can hide the 
meaning and interest of our results.

\medskip

A previous study on zeta and M\"{o}bius duality is
found in \cite{LS}. One of its results is what we called
conditions $(i)$ in Propositions \ref{prop1} and \ref{prop2}
in Section \ref{Sec2}, we give them for completeness and because
they are straightforward to obtain. The main result in 
\cite{LS} is Theorem 2, ensuring that 
there exists a strong dual (see \cite{DF}) 
for a stochastic kernel $P$  
such that the ratio between the initial distribution and 
the stationary distribution is M\"{o}bius monotone but also 
(mainly) that time reversed process is  M\"{o}bius monotone.
This type of questions will not be in the focus of our work.

\subsection{Notation}
For a set $A$, $|A|$ denotes its cardinality.
By $I, \, \I$ we denote finite sets. 
We denote
$\bP(I)=\{J: J \subseteq I\}$ the class of subsets of $I$.

\medskip

By $N$, $T$ we mean positive integers.
We set $\aI_N=\{1,..,N\}$. For two integers
$s\le t$ we denote by $\aI_t^s=\{s,..,t\}$ the interval of integers.
In particular $\aI_N^0=\{0,1,..,N\}$. 

\medskip 

For a relation $\R$ defined on some set, we define $\1_\R$ the
function which assigns a $1$ when $\R$ is satisfied and 
$0$ otherwise. For a set $A$, $\1_A$ is its characteristic function,
it gives value $1$ for the elements belonging to $A$ and $0$ otherwise.
Also we denote by $\1$ the $1-$constant vector with the 
dimension of the space where it is defined.   

\medskip

The transpose of a matrix or a vector $H$ is denoted by $H'$.
The functions $g:\I\to \RR$ can be identified to a column vector in
$\RR^\I$, so $g'$ means the row vector. In particular the 
characteristic function $\1_A$ is a column vector and $\1'_A$ a row vector.

\section{Zeta and M\"{o}bius matrices}
\label{SecZM}
This section follows the ideas developed by Rota in \cite{GR}.
The examples we give are well-known and the product formula supplied 
in \cite{BG} allows to present them in a unified way.

\medskip

Let $\I$ be a finite set and $(\I, \preceq)$ be a 
partially ordered space.

\medskip

The zeta matrix $Z=(Z(a,b): a,b\in \I)$
is given by $Z(a,b)= \1_{a\preceq b}$.
It is nonsingular and its inverse $Z^{-1}=(Z^{-1}(a,b): a,b\in \I)$
is the M\"{o}bius matrix. In \cite{BG} it was shown that 
the M\"{o}bius matrix satisfies $Z^{-1}(a,b)= \mu(a,b) \1_{a\preceq b}$, 
where for $a\preceq b$:
\begin{equation}
\label{mu1}
\mu(a,b)=
\begin{cases}
&1 \hbox{ if } a=b\\
& -\sum_{c\in \I: a\preceq c \preceq b} \mu(a,c) \hbox{ if }
a\prec b\,.
\end{cases}
\end{equation}
Also see \cite{GR} Section $3$.
For completeness, let us check that this matrix is
the inverse of $Z$. We have
$$
\sum_{c\in \I} \1_{a\preceq c} \mu(a,c) \1_{c\preceq b}=
\sum_{c\in \I: a\preceq c\preceq b} \mu(a,c).
$$
If $a=b$ then $c=a=b$ is the unique $c$ in the sum
and the above expression is $1$. When
$a\neq b$, in order that there exists some $c$ in the sum we must have
$a\prec b$. In this case, by definition of $\mu$ we have
$$
\left(\sum_{c\in \I: a\preceq c\prec b} \mu(a,c)\right)+\mu(a,b)=0.
$$
so, the inverse of $Z$ satisfies 
$Z^{-1}(a,b)= \mu(a,b) \1_{a\preceq b}$.
The function $\mu(a,b)$, that only needs to be defined for 
$a\preceq b$, is called the M\"{o}bius function. 
Since $\mu(a,a)=1$, $\mu$ is completely described once 
$\mu(a,b)$ is identified for $a\prec b$.

\medskip

We will also consider the transpose zeta and M\"{o}bius matrices
$Z'=(Z'(a,b)= \1_{b\preceq a}: a,b\in \I)$ and 
${Z'}^{-1}=({Z'}^{-1}(a,b)=\mu(b,a) \1_{b\preceq a}: a,b\in \I)$.

\medskip

Two partially ordered spaces $(\I_1, \preceq_1)$ $(\I_2, \preceq_2)$
are isomorphic if there exists a bijection $\varphi:\I_1\to \I_2$
that verifies $a \preceq_1 b$ if and only if
$\varphi(a) \preceq_2 \varphi(b)$. If $\mu_1$ and $\mu_2$
are their respective M\"{o}bius functions, then
$\mu_1(a,b)=\mu_2(\varphi(a),\varphi(b))$.

\medskip

\subsection{Product formula}
Let us introduce the product formula, as given in Theorem $3$ in \cite{GR}. 
Let $(\I_1,\preceq_1)$ and $(\I_2,\preceq_2)$ be two partially ordered spaces
with M\"{o}bius functions $\mu_1$ and $\mu_2$
respectively. The product set $\I_1 \times \I_2$ is partially ordered
with the product order $\preceq_{1,2}$ given by:
$(a_1,a_2)\preceq_{1,2} (b_1,b_2)$ if 
$a_1\preceq_1 b_1$ and $a_2\preceq_2 b_2$. The
M\"{o}bius function for the product space 
$(\I_1 \times \I_2,\preceq_{1,2})$ results to be the 
product of the M\"{o}bius functions:
\begin{equation}
\label{mu2}
a_1\preceq_1 b_1,\, a_2\preceq_2 b_2 \, \Rightarrow \, 
\mu((a_1,a_2),(b_1,b_2))=\mu_1(a_1,b_1)\mu_2(a_2,b_2). 
\end{equation}
The above relations are summarized in,
\begin{eqnarray}
\nonumber
&{}&\1_{(a_1,a_2)\preceq_{1,2} (b_1,b_2)}=\1_{a_1\preceq_{1} b_1}
\1_{a_2\preceq_{2} b_2} \,;\\
\label{pro1}
&{}&\1_{(a_1,a_2)\preceq_{1,2} (b_1,b_2)}\mu((a_1,a_2),(b_1,b_2))=
\mu_1(a_1,b_1)) \1_{a_1\preceq_{1} b_1}\cdot \mu_2(a_2,b_2) \1_{a_2\preceq_{2} b_2}.
\end{eqnarray}
Let $Z_r$ be the zeta matrix associated to $(\I_r,\preceq_r)$ for $r=1,2$,
and  $Z_{1,2}$ be the zeta matrix associated to 
the product space $(\I_1 \times \I_2,\preceq_{1,2})$.
For $g_r:\I_r\to \RR$ for $r=1,2$ define 
$g_1\otimes g_2:\I_1 \times \I_2\to \RR$ by
$g_1\otimes g_2(a_1,a_2)=g_1(a_1)g_2(a_2)$. By using (\ref{pro1}) we get 
\begin{eqnarray}
\nonumber
&{}&
(Z_{1,2} \, g_1\otimes g_2) ((a_1,a_2))=(Z_1 g_1)(a_1) (Z_2 g_2)(a_2)
\,;\\
\label{pro2}
&{}& (Z_{1,2}^{-1}\, g_1\otimes g_2) ((a_1,a_2))=(Z^{-1}_1 g_1)(a_1) 
(Z^{-1}_2 g_2)(a_2).
\end{eqnarray}

\subsection{M\"{o}bius functions for sets}
\label{SecZMse}
The most trivial case is $|\I|=2$. Take $\I=\{0,1\}$ with the usual
order $\le$. In this case $\mu(0,1)=-1$.
Then, the M\"{o}bius function of the product space $\{0,1\}^I$ endowed 
with the product partial order $\le$ is
\begin{equation}
\label{prx1}
\mu((a_i: i\in I),(b_i: i\in I))=(-1)^{\sum_{i\in I} (a_i-b_i)}
\hbox{ when } (a_i: i\in I)\le (b_i: i\in I).
\end{equation}
Let $I$ be a finite set, the class of its subsets 
$\bP(I)=\{J: J \subseteq I\}$ is partially ordered by 
inclusion $\subseteq$. Since $(\bP(I),\subseteq)$ 
is isomorphic to the product space $\{0,1\}^{I}$ endowed with the product
partial order, the M\"{o}bius function for $(\bP(I), \subseteq)$ is
\begin{equation}
\label{prx2}
\forall J,K\in \bP(I), J\subseteq K: \quad \mu(J,K)=(-1)^{|K|-|J|}.
\end{equation}
Its zeta matrix $Z=(Z(J,K): J,K\in \bP(I))$
satisfies $Z(J,K)= \1_{J\subseteq K}$ and
the M\"{o}bius matrix $Z^{-1}$ is given by
$Z^{-1}(J,K)=(-1)^{|K|-|J|} \1_{J\subseteq K}$.
The transpose matrices $Z'$ and ${Z^{-1}}'$ satisfy
$Z'(J,K)= \1_{K\subseteq J}$ and 
${Z^{-1}}'(J,K)=(-1)^{|J|-|K|} \1_{K\subseteq J}$.

\medskip

Let $T\ge 1$ be a positive integer.
The study of  $(\bP(I),\subseteq)$ also encompasses the class
of product of sets $\bP(I)^T$ endowed with the product order. 
To describe it, denote the elements of $\bP(I)^T$ by
$$
{\vec J}=(J_t: t\in \aI_T) \hbox{ with } J_t\subseteq I \hbox{ for } t\in \aI_T.
$$
Let ${\vec J}$ and ${\vec K}$ be two elements of
$\bP(I)^T$. The product order is ${\vec J}\subseteq {\vec K}$
if $J_t\subseteq K_t$ for $t\in \aI_T$. The  M\"{o}bius function for the
product ordered space $(\bP(I)^T,\subseteq)$ is 
\begin{equation}
\label{mmu3}
\mu({\vec J},{\vec K})=(-1)^{\sum_{t\in \aI_T}(|K_t|-|J_t|)} \hbox{ when }
{\vec J}\subseteq {\vec K}.
\end{equation}

Now note that 
\begin{equation}
\label{prx3}
\bP(I)^T\to \bP(I\times \aI_T), \; (J_t: t\in \aI_T)\to  
\bigcup_{t \in T} J_t\times \{t\}\,, 
\end{equation}
is a bijection that satisfies
$({\vec J}\subseteq {\vec K})\, \Leftrightarrow \, 
(\bigcup_{t \in T} J_t\times \{t\}\subseteq \bigcup_{t \in T} 
K_t\times \{t\})$. 
Then, the above bijection is an isomorphism
between the partially ordered spaces 
$(\bP(I)^T,\subseteq)$ and $(\bP(I\times \aI_T),\subseteq)$.
Hence, every statement for the class of sets also holds for 
the class of product of sets (the isomorphism between both spaces 
is a natural consequence of the construction done between 
(\ref{prx1}) and (\ref{prx2})).

\subsection{M\"{o}bius functions for partitions}
\label{SecZMpa}
Let $I$ be a finite set and $\aP(I)$ be the set of partitions of $I$. Thus,
$\alpha\in \aP(I)$ if $\alpha=\{A_t: t=1,..,T(\alpha)\}$, where:
$$
\forall t\in \aI_{T(\alpha)}\;\, A_t\! \in \!\bP(I)\setminus \{\emptyset\}, \; 
\, t\neq t' \; A_t\cap A_{t'}\!=\!\emptyset \hbox{ (disjointedness)},\;
\bigcup_{t\in \aI_{T(\alpha)}}\!\!\! A_t\!=\!I \hbox{ (covering)}.
$$
The sets $A_t$ are called the atoms of the partition, and 
the number of atoms constituting the partition $\alpha$ is denoted by 
$[\alpha]=T(\alpha)$.
An atom of $\alpha$ is often denoted by $A$ and we write
$A\in \alpha$. Since the order of the atoms plays no role
we write $\alpha=\{A\in \alpha\}$

\medskip

A partition $\alpha$ can be defined as the set of equivalence classes of 
an equivalence relation ${\equiv}_\alpha$ defined by $i{\equiv}_\alpha j \, 
\Leftrightarrow \, \exists A\in \alpha$  such that $i,j\in A$. 
That is, two elements are in relation ${\equiv}_\alpha$ when they are 
in the same atom of the partition.

\medskip

The set of partitions $\aP(I)$ is partially ordered by the 
following order relation
$$
\alpha\preceq \beta \hbox{ if } \forall \, A\in \alpha \,
\exists \, B\in \beta \hbox{ such that } A\subseteq B. 
$$
When $\alpha\preceq \beta$ it is said that $\alpha$ is finer 
than $\beta$ or that $\beta$ is coarser than $\alpha$. 

\medskip

The zeta matrix $Z=(Z(\alpha,\beta): \alpha,\beta\in \aP(I))$ is
given by $Z(\alpha,\beta)=\1_{\alpha\preceq \beta}$
and the M\"{o}bius matrix by $Z^{-1}(\alpha,\beta)= 
\mu(\alpha,\beta) \1_{\alpha\preceq \beta}$.
The M\"{o}bius function $\mu(\alpha,\beta)$ is shown to
satisfy the relation 
$$
\mu(\alpha,\beta)=(-1)^{[\alpha]+[\beta]}\prod_{B\in \beta} (\ell_B^\alpha-1)!
\hbox{ for } \alpha\prec \beta\,,
$$
where $\ell_B^\alpha=|\{A\in \alpha: A\subseteq B\}|$
is the number of atoms of $\alpha$ contained in $B$, see \cite{C} p. 36.

\section{Zeta and M\"{o}bius Duality}
\label{Sec2}
We will study duality relations for zeta and M\"{o}bius matrices
and the conditions for positivity in terms of what we 
call M\"{o}bius positive cones. Here, $\I$ is the set of indexes 
and as assumed it is finite.

\subsection{Duality}
\label{SecDual}
Let $P=(P(a,b): a,b\in \I)$ be a positive matrix, that is each entry is
non-negative, and $H=(H(a,b): a,b\in \I)$ be a matrix. Then,
$Q=(Q(a,b): a,b\in \I)$ is said to be a $H-$dual of $P$ if it satisfies
\begin{equation}
\label{dual0}
H Q'=P H\,.
\end{equation}
We usually refer to $P$ and $Q$ as kernels, and $Q$ is said to be the
dual kernel.
Duality relation (\ref{dual0}) implies $H Q'^n=P^n H$ for all $n\ge 0$.
If $H$ is nonsingular the duality relation (\ref{dual0}) takes the form
\begin{equation}
\label{dual1}
Q'=H^{-1} P H\,.
\end{equation}
One is mostly interested in the case when $P$ is substochastic
(that is nonnegative and satisfying $P\1\le \1$) or
stochastic (nonnegative and $P\1=\1$) and one looks for conditions in order
that $Q$ is nonnegative and, when this is the case, one seeks to know
when $Q$ is substochastic or stochastic.

\medskip

Now, let $h:\I\to \RR_+$ be a non-vanishing function 
and $D_h$ be the diagonal matrix given by $D_h(a,a)=h(a)$ 
for $a\in \I$. Its inverse is $D_h^{-1}=D_{h^{-1}}$.

\begin{lemma}
\label{lem2}
Let $h:\I\to \RR_+$ be a non-vanishing function. We have:
\begin{equation}
\label{dual3}
H Q'=P H \, \Leftrightarrow \, H_h Q'_{h^{\!-\!1}\!,h}=P H_h \hbox{ with }
H_h:=H D_h^{-1}\hbox{ and } Q_{h^{\!-\!1}\!,h}:= D_h^{-1}Q D_h\,.
\end{equation}
Assume $h>0$. Then, $Q\ge 0$ implies $Q_{h^{\!-\!1}\!,h}\ge 0$ and
\begin{equation}
\label{dual4}
\left(Q_{h^{\!-\!1}\!,h}\1=\1 \Leftrightarrow Qh=h\right)
\hbox{ and }
\left(Q_{h^{\!-\!1}\!,h}\1\le \1 \Leftrightarrow Qh\le h\right).
\end{equation}
\end{lemma}

\begin{proof}
All relations are straightforward. For instance (\ref{dual4}) follows
from $Q_{h^{\!-\!1}\!,h}\1= \1$ if and only if $Q D_h\1= D_h\1$,  
which is $Q h=h$. A similar argument proves the second relation with 
$\le$.
\end{proof}

The matrix $Q_{h^{\!-\!1}\!,h}$ is called the $h-$transform of $Q$.
So, it is a $H_h-$dual of $P$. When $h>0$ and $Q\ge 0$, 
the matrix $Q_{h^{\!-\!1}\!,h}$ is stochastic
if and only if $h$ is a right eigenvector of $Q$ with eigenvalue $1$.

\medskip

If $P$ and $Q$ are substochastic matrices
then the duality has the following probabilistic
interpretation in terms of their associated Markov chains.
Let $X=(X_n: n\ge 0)$ and $Y=(Y_n: n\ge 0)$ be the associated Markov
chains and $\T^X$ and $\T^Y$ be their lifetimes. Let
$\PP^X_a$ and $\PP^Y_b$ be the laws of the chains starting from the states
$a$ and $b$ respectively, and $\EE^X_a$ and $\EE^X_b$ be their
associated mean expected values.
Let $\partial^X$ and $\partial^Y$ be the coffin states of $X$ and $Y$
respectively, then $X_n=\partial^X$ for $n\ge \T^X$ and
$Y_n=\partial^Y$ for $n\ge \T^Y$. We make the extension
$$
H(\partial^X,b)=0=H(a,\partial^Y)=
H(\partial^X,\partial^Y)\,.
$$
Then, the duality relation (\ref{dual0}) is equivalent to
$$
\forall a,\, b\in \I \;\, \forall n\ge 0:\quad
\EE^X_a(H(X_n,b)= \EE^Y_b(H(a,Y_n))\,.
$$

This notion of duality was introduced in \cite{Lig} in a very general framework
and developed in several works, see \cite{DF, HM, KL, MM1} 
and references therein.

\subsection{M\"{o}bius positive cones}
\label{SecCones}
We will study duality relations for zeta and M\"{o}bius matrices and
set conditions for positivity in terms of the following classes
of nonnegative functions
$$
\F_+(\I)=\{g\in \RR_+^{\I}: Z^{-1} g\ge 0\}
\, \hbox{ and } \,
\F'_+(\I)=\{g\in \RR_+^{\I}: {Z^{-1}}' g\ge 0\} \,,
$$
Note that both sets are convex cones, we call them positive M\"{o}bius cones
(of functions). We have
$$
\F_+(\I)=\{g\in \RR^{\I}: Z^{-1} g\ge 0\}
\, \hbox{ and } \,
\F'_+(\I)=\{g\in \RR^{\I}: {Z^{-1}}' g\ge 0\} \,.
$$
For showing the first expression we only have to prove that if
$Z^{-1} g\ge 0$ then $g\ge 0$. This follows straightforward
from the non-negativity of $Z$,
$$
\forall a\in \I:\;\;
g(a)=\sum_{b: b\succeq a} (Z^{-1} g)(b)\,.
$$
The second expression is shown similarly. 
In  \cite{LS} the functions in $\F_+(\I)$ and 
$\F'_+(\I)$ are called M\"{o}bius monotone
and the argument we just gave is the Proposition $2.1$ therein.
We also define
$$
\F(\I)=\F_+(\I)-\F_+(\I)=\{g_1-g_2: g_1,g_2\in \F_+(\I)\} \, 
\hbox{ and } \, \F'(\I)=\F'_+(\I) - \F'_+(\I).
$$

For every $a\in \I$ the function $\RR^{\I}\to \RR$, $g\to Z^{-1} g (a)$ 
is linear. Hence, a simple consequence of the additivity gives
\begin{eqnarray}
\label{in1}
&{}&\forall \, g_1, \, g_2\in \F_+(\I)\,, a\in \I \;\;\;
\Rightarrow \, Z^{-1} g_1(a)\le Z^{-1} (g_1+g_2)(a)\,;\\
\label{in2}
&{}&\forall \, g_1, \, g_2\in \F'_+(\I)\,, a\in \I \;\;\;
\Rightarrow \, {Z^{-1}}' g_1(a)\le {Z^{-1}}' (g_1+g_2)(a)\,.
\end{eqnarray}

\subsection{Duality with Zeta and M\"{o}bius matrices}
\label{SecZM}

We will give necessary and sufficient conditions 
in order that zeta and M\"{o}bius duality, 
as well as their transpose, preserve positivity (these conditions 
appear as $(i)$ in the propositions).
Also we give stronger sufficient conditions having
stronger implications on the monotonicity of kernels
(these conditions appear as $(ii)$ in the propositions).

\medskip

As said, zeta and M\"{o}bius duality were already studied in  \cite{LS} 
and in this reference conditions $(i)$ of Propositions 
\ref{prop1} and \ref{prop2} are also found. We supply them for completeness 
and since they are straightforward.  

\medskip

In the sequel, we will introduce a notation for the rows and 
columns of a matrix.
For $P=(P(a,b):a,b\in \I)$ we denote by $P(a,\bullet)$
its $a-$th row and by $P(\bullet,b)$ its $b-$th column, that is
$$
P(a,\bullet):\I\to \RR,\; c\to P(a,c) \hbox{ and }
P(\bullet,b):\I\to \RR, \; c\to P(c,b).
$$

\subsubsection{Duality with the zeta matrix}
Assume the kernel $Q$ is the $Z-$dual of the positive kernel
$P$, so $Q'=Z^{-1} P Z$ holds. Hence,
\begin{eqnarray}
\nonumber
Q(a,b)&=&\sum_{c\in \I}\sum_{d\in \I}
Z^{-1}(b,c)P(c,d) Z(d,a)=
\sum_{c: b\preceq c}\mu(b,c)\sum_{d: d\preceq a}P(c,d)\\
\label{es1}
&=&\sum_{c: b\preceq c}\mu(b,c)
\left(\sum_{d: d\preceq a}P(\bullet, d)\right)(c)=
Z^{-1}\left(\sum_{d: d\preceq a}P(\bullet, d)\right)(b).
\end{eqnarray}

\medskip

\begin{proposition}
\label{prop1}
Assume $P\ge 0$.
 $(i)$ We have
\begin{equation}
\label{equi3}
Q\ge 0 \, \Leftrightarrow \,  \forall a\in \I\,: \;\,
\sum\limits_{d: d\preceq a}P(\bullet,d)\in \F_+(\I).
\end{equation}
When this condition holds the following implication is satisfied,
\begin{equation}
\label{fac1}
\left(P(c,d)>0 \; \Rightarrow \; c\preceq d\right) \hbox{ implies }
\left(Q(c,d)>0 \; \Rightarrow \; d\preceq c\right).
\end{equation}

\noindent $(ii)$ Assume for all $d\in \I$ we have $P(\bullet,d)\in \F_+(\I)$.
Then $Q\ge 0$ and for all $b$ the function $Q(a,b)$ is increasing in $a$,
that is
\begin{equation}
\label{elincs3}
\forall b\in \I, \, a_1\preceq a_2 \; \Rightarrow \;
Q(a_1,b) \le Q(a_2,b)\,.
\end{equation}
\end{proposition}

\begin{proof}
The equivalence (\ref{equi3}) is straightforward from equality (\ref{es1}).
To show relation (\ref{fac1}) we use the equality
$$
Q(a,b)=\sum_{(c,d): b\preceq c, \; d\preceq a}
\mu(b,c) P(c,d).
$$
Since we are assuming $P$ only charges couples $(c,d)$ such that
$c\preceq d$ then the previous sum is with respect to the
set $\{(c,d): b\preceq c; \; d\preceq a; \; c\preceq d\}$.
So, if this set is nonempty we necessarily have $b\preceq a$.

\medskip

\noindent $(ii)$ The first statement follows from $(i)$ and
the fact that $\F_+(\I)$ is a cone. For proving (\ref{elincs3}) we note that
$a_1\preceq a_2$ implies $\{d: d\preceq a_2\}\supseteq \{d: d\preceq a_1\}$.
Then,
$$
\sum_{d\preceq a_2} P(\bullet, d) =\sum_{d\preceq a_1} P(\bullet, d)+ g
\hbox{ with }
g=\sum_{d: d\preceq a_2, d\not\preceq a_1} P(\bullet, d).
$$ 
Then, from the hypothesis made in $(ii)$ we get $g\in \F_+(\I)$.
Hence, (\ref{in1}) and (\ref{es1}) give (\ref{elincs3}).
\end{proof}

\begin{remark}
\label{rem1}
Assume condition (\ref{equi3}) is satisfied
and that $(\I,\preceq)$ has a global maximum and a global minimum,
denoted respectively by $a_{max}$ and $a_{min}$.
Then, the hypothesis
$\left(P(c,d)>0 \; \Rightarrow \; c\preceq d\right)$
in (\ref{fac1}) assumes in particular
that $a_{max}$ is an absorbing point for $P$ because $P(a_{max},b)=0$
for all $b\neq a_{max}$. The property that it implies,
$\left(Q(c,d)>0 \; \Rightarrow \; d\preceq c \right)$, says in particular
that $a_{min}$ is an absorbing point for
$Q$ because $Q(a_{min}, d)=0$ for all $d\neq a_{min}$.
In the case 
$(\I,\preceq)=(\bP(I),\subseteq)$ we have $a_{max}=I$ and  $a_{min}=\emptyset$
and when $(\I,\preceq)=(\aP(I),\preceq)$ we have 
$a_{max}=\{I\}$ and  $a_{min}=\{\{i\}: i\in I\}$. 
\end{remark}

\begin{remark}
\label{rem2}
Under hypothesis $(ii)$, condition (\ref{elincs3}) implies
that if $Q$ is stochastic
then for comparable indexes the rows of $Q$ are equal.
\end{remark}

\subsubsection{Duality with the transpose zeta matrix}
Let the kernel $Q$ be the $Z'-$dual of the positive kernel
$P$, so $Q'={Z'}^{-1} P Z'$
is satisfied. Hence,
\begin{equation}
\label{es1x}
Q(a,b)=\sum_{c\in \I}\sum_{d\in \I}{Z}^{-1}(c,b) P(c,d) Z(a,d)
={Z^{-1}}'\left(\sum_{d: a\preceq d}P(\bullet, d)\right)(b).
\end{equation}

\medskip

\begin{proposition}
\label{prop2}
 $(i)$ We have
\begin{equation}
\label{equi4}
Q\ge 0 \, \Leftrightarrow \,  \forall a\in \I\,: \;\,
\sum\limits_{d: a\preceq d}P(\bullet,d)\in \F'_+(\I).
\end{equation}
When this condition holds the following implication is satisfied
\begin{equation}
\label{fac2}
\left(P(c,d)>0 \; \Rightarrow \; d\preceq c\right) \hbox{ implies }
\left(Q(c,d)>0 \; \Rightarrow \; c\preceq d\right).
\end{equation}

\noindent $(ii)$ Assume for all $d\in \I$ we have $P(\bullet,d)\in \F'_+(\I)$.
Then $Q\ge 0$ and for all $b$ the function $Q(a,b)$ is increasing in $a$,
that is
\begin{equation}
\label{dec4}
\forall b\in \I, \, a_2\preceq a_1 \; \Rightarrow \;
Q(a_1,b) \ge Q(a_2,b)\,.
\end{equation}
\end{proposition}

\begin{proof}
It is entirely similar as the one of Proposition \ref{prop1}.
\end{proof}

Similar notes as Remarks \ref{rem1} and \ref{rem2} can be 
made.

\medskip

The conditions in part $(i)$ of Propositions \ref{prop1} and
\ref{prop2} ensuring positivity of $Q$ are the same as the ones in \cite{LS}.

\subsubsection{Duality with the M\"{o}bius matrix }

Assume $Q$ is the $Z^{-1}-$dual of the positive kernel
$P$, so $Q'=Z PZ^{-1}$
is satisfied. This is
\begin{equation}
\label{s1}
Q(a,b)=\sum_{c\in \I} \, \sum_{d\in \I}Z(b,c)
P(c,d)Z^{-1}(d,a)
={Z^{-1}}'\left(\sum\limits_{c: b\preceq c}\!\!P(c,\bullet) \right)(a).
\end{equation}

\medskip

\begin{proposition}
\label{prop3}
Assume $P\ge 0$.

\noindent $(i)$ We have
\begin{equation}
\label{equi1}
Q\ge 0 \, \Leftrightarrow \,  \forall b\in \I\,: \;\,
\sum\limits_{c: b\preceq c}\!\!P(c,\bullet)\in \F'_+(\I).
\end{equation}
If this condition holds we have that (\ref{fac1}) is satisfied.

\medskip

\noindent $(ii)$ Assume for all $c\in \I$ we have $P(c,\bullet)\in \F'_+(\I)$. Then
$Q\ge 0$ and:

\noindent $(ii1)$ $Q(a,b)$ is decreasing in $b$, that is
\begin{equation}
\label{dec1}
\forall a\in \I, \, b_1\preceq b_2 \; \Rightarrow \;
Q(a,b_1) \ge Q(a,b_2)\,;
\end{equation}

\noindent $(ii2)$ If $P$ is stochastic and irreducible then its
invariant distribution $\rho$ satisfies $\rho\in \F'_+(\I)$;

\medskip

\noindent $(ii3)$ If $Q$ is stochastic and irreducible then its
invariant distribution ${\widehat \rho}$ is decreasing that is:
$b_1\preceq b_2 \; \Rightarrow \; {\widehat \rho}(b_1) \ge
{\widehat \rho}(b_2)$.

\end{proposition}

\begin{proof}
The proof of $(i)$, the first statement in $(ii)$ and $(ii1)$ 
are similar to the proof of Proposition \ref{prop1}.
 
\medskip

\noindent $(ii2)$ The invariant distribution $\rho=(\rho(a): a\in \I)$
satisfies $\rho'=\rho' P$, so in our notation
$\rho=\sum_{a\in \I} \rho(a) P(a,\bullet)$. From our hypothesis we have that
$P(a,\bullet)\in \F_+(\I)$ for all $a\in \I$; since 
$\F_+(\I)$ is a cone we get the result.

\medskip

$(ii3)$ Since ${\widehat \rho}=\sum_{a\in \I}{\widehat \rho}(a) 
Q(a,\bullet)$ the property is derived from
property $(ii1)$.
\end{proof}

A similar note as Remark \ref{rem1} can be made. Duality with the 
M\"{o}bius matrix is a special case of duality with 
non-positive matrices. For a study considering other 
non-positive duality matrices see \cite{SL}.

\subsubsection{Duality with the transpose M\"{o}bius matrix }

Assume $Q$ is ${Z^{-1}}'-$dual of the positive matrix
$P$, so $Q'=Z' P {Z^{-1}}'$
is satisfied. Then,
\begin{equation}
\label{s2}
Q(a,b)=\sum_{c\in \I} \sum_{d\in \I} Z(c,b)P(c,d)Z^{-1}(a,d)
{Z^{-1}}\left(\sum_{c: c\preceq b}\!\!P(c,\bullet)\right)(a).
\end{equation}

\medskip

\begin{proposition}
\label{prop4}
Assume $P\ge 0$.
 $(i)$ We have
\begin{equation}
\label{equi2}
Q\ge 0 \, \Leftrightarrow \,  \forall b\in \I\,: \;\,
\sum\limits_{c: c\preceq b}P(c,\bullet)\in \F_+(\I).
\end{equation}
When this condition holds, relation (\ref{fac2}) is satisfied.

\medskip

\noindent $(ii)$ Assume for all $c\in \I$ we have $P(c,\bullet)\in \F_+(\I)$.
Then $Q\ge 0$ and:

\medskip

\noindent $(ii1)$ $Q(a,b)$ is increasing in $b$, this is
\begin{equation}
\label{inc1x}
\forall a\in \I, \, b_1\preceq b_2 \; \Rightarrow \;
Q(a,b_1) \le Q(a,b_2)\,;
\end{equation}

\noindent $(ii2)$ If $P$ is stochastic and irreducible then its
invariant distribution $\rho$ satisfies $\rho\in \F_+(\I)$;

\medskip

\noindent $(ii3)$ If $Q$ is stochastic and irreducible then its
invariant distribution ${\widehat \rho}$ is increasing.

\end{proposition}

\begin{proof}
The proof of $(i)$, the first statement in $(ii)$ and $(ii1)$
are similar to the proof of Proposition \ref{prop1}, and
the parts $(ii2)$ are $(ii3)$ are shown in a 
similar way as $(ii2)$ and $(ii3)$ in Proposition \ref{prop3}.
\end{proof}

A similar note as Remark \ref{rem1} can be made.

\section{M\"{o}bius positive cones and Sylvester formulae 
for sets and partitions}
\label{Sec3}

\subsection{Sylvester formulae}
\label{SecSylv}

As already fixed $I$ is a finite set.
Let $(\X,\B)$ be a measurable space and 
$(X_i: \,i\in I)\subseteq \B$ be a finite class of events. 
The $\sigma-$algebra $\sigma(X_i: \,i\in I)$ generated by
$(X_i: \,i\in I)$ in $\X$, is the class of finite
unions of the disjoint sets
\begin{equation}
\label{atom1}
\bigcap\limits_{i\in J}X_i \setminus \left(\bigcup\limits_{L: L\supseteq J, \, 
L\neq J} \; \bigcap\limits_{i\in L}X_i\right), \; J\subseteq I\,.
\end{equation}
When $J=\emptyset$ the above set is $\X\setminus \bigcup_{i\in I} X_k$
because $\bigcap\limits_{i\in \emptyset}X_i=\X$.

\medskip

Since all we shall do only depends on 
$\sigma(X_i: \,i\in I)$, in the sequel
we only consider the measurable space $(\X,\sigma(X_i: \,i\in I))$.
When we say $(\X,\sigma(X_i: \,i\in I))$ is
a measurable space we mean $(X_i: \,i\in I)$ is a
family of subsets of $\X$ and $\sigma(X_i: \,i\in I)$ is the $\sigma-$algebra
generated by them.

\subsubsection{Sylvester formula for sets and product of sets}
\label{SecSylv1}
Let $(\X,\sigma(X_i: \,i\in I))$ be a measurable space.
Let $\nu$ be a finite measure or finite signed measure on 
$(\X,\sigma(X_i: \,i\in I))$. Sylvester formula is
$\nu(\X\setminus \bigcap\limits_{i\in I} X_i)=
\sum_{L\subseteq I} (-1)^{|L|} \nu(\bigcap\limits_{i\in L}X_i)$.
Let $J\in \bP(I)$ be fixed and consider
$\X'=\bigcap\limits_{i\in J}X_i$ and $X'_i=X_i\bigcap\limits \X'$
for $i\in I\setminus J$. We have
$\X'\setminus X'_k=\bigcap\limits_{i\in J}X_i\setminus
\bigcap\limits_{i\in J\cup \{j\}} X_i$
for all $j\in I\setminus J$. Then,
$\bigcap\limits_{i\in I\setminus J} \X'\setminus X'_i=
\bigcap\limits_{i\in J}X_i \setminus \! (\bigcup\limits_{L: L\supseteq J, \, L\neq J}
\; \bigcap\limits_{i\in L}X_i)$, and Sylvester formula gives
\begin{equation}
\label{syl1}
\forall J\!\in \! \bP(I):\;
\nu\left(\bigcap\limits_{i\in J}X_i \setminus
\! \left(\bigcup_{L: L\supseteq J, \, L\neq J}
\; \bigcap\limits_{i\in L}X_i\right)\!\right)\!=\!
\! \sum_{L: L\supseteq J}\! (-1)^{|L|-|J|} \nu(\bigcap\limits_{i\in L}X_i).
\end{equation}
We will write the above formula in terms of the M\"{o}bius matrix for sets.

\begin{proposition}
\label{prop5}
The measurable spaces $(\X,\sigma(X_i: \,i\in I))$ and
$(\bP(I),\bP(\bP(I))$ are isomorphic by: 
\begin{equation}
\label{iso1}
\Psi:\sigma(X_i\!: \!i\!\in \!I)\to \bP(\bP(I)):\,\;
\forall\, J\!\in \! \bP(I),\,
\Psi\left(\bigcap\limits_{i\in J}X_i \setminus \! 
\left(\bigcup_{L: L\supseteq J, \, L\neq J}
\; \bigcap\limits_{i\in L}X_i\right)\!\right)=\{J\}.
\end{equation}
For the other elements of the algebras we impose that $\Psi$ preserves
disjoint unions; thus $\Psi$ is an isomorphism of algebras.

\medskip

For every finite (respectively signed) measure $\nu$ defined on
$(\X,\sigma(X_i: \,i\in I))$ the (respectively
signed) measure $\nu^*=\nu\circ \Psi^{-1}$ on $(\bP(I),\bP(\bP(I))$
is given by:
\begin{equation*}
\forall\, J\in \bP(I):\;\; \nu^*(\{J\})=
\nu\left(\bigcap\limits_{i\in J}X_i \setminus 
\left(\bigcup_{L: L\supseteq J, \, L\neq J}
\; \bigcap\limits_{i\in L}X_i\right)\right)\,.
\end{equation*}
Under the isomorphism (\ref{iso1}) we have
\begin{equation}
\label{re0a}
\Psi\left(\bigcap\limits_{i\in J}X_i\right)=\{K: K\!\supseteq \!J\}\,;
\end{equation}
\begin{equation}
\label{re0b}
\forall i\in I:\; \Psi(X_i)=\{J: i\!\in \!J\}\,;
\end{equation}
\begin{equation}
\label{re0c}
\nu^*(\{J\})=\!\sum_{L: L\supseteq J}\!\! (-1)^{|L|-|J|} \,
\nu(\bigcap\limits_{i\in L}X^*_i)=\!\sum_{L: L\supseteq J}\!\! 
(-1)^{|L|-|J|} \left(\sum_{K: K\supseteq L}\!\!\! \nu^*(\{K\})\!\right).
\end{equation}
\end{proposition}

\begin{proof}
Let $X^*_i=\Psi(X_i)$ be the image of $X_i$ under this isomorphism, so
$\bigcap\limits_{i\in J}X^*_i=\Psi\left(\bigcap\limits_{i\in J}X_i\right)$.
Since
$$
\bigcap\limits_{i\in J}X_i=\bigcup\limits_{K\supseteq J}\left(
\bigcap\limits_{i\in K}X_i \setminus \left(\bigcup_{L: L\supseteq K, \, L\neq K}
\; \bigcap\limits_{i\in L}X_i\right)\right).
$$
the isomorphism gives (\ref{re0a}).
Then
$$
X^*_i=\bigcup\limits_{J: i\in J} \, \bigcap\limits_{j\in J}X^*_j=
\bigcup\limits_{J: i\in J} \{K: K\!\supseteq \!J\}=\{J: i\!\in \!J\},
$$
so (\ref{re0b}) is shown. Then,
(\ref{re0c}) follows from Sylvester formula (\ref{syl1}).
\end{proof}

Note that (\ref{re0c}) is equivalent to
$\nu^*(\{J\})=({Z}^{-1}(Z\nu^*))(\{J\})$ when $\nu^*=(\nu^*(\{J\}: J\in \bP(I)$ 
is written as a column vector.
Hence, Sylvester formula (\ref{syl1}) is equivalent to the fact that 
the M\"{o}bius function for the class of subsets ordered by inclusion
is $(-1)^{|L|-|J|}$ for $J\subseteq L$.

\medskip

As noted, the isomorphism given in (\ref{prx3}) guarantees that
a similar Sylvester formula can be stated for product of sets. 
Let us give this formula explicitly.
Let $T\ge 1$ be a positive integer. The product space $\bP(I)^T$ 
was endowed with the product order also denoted by 
$\subseteq$, the elements of $\bP(I)^T$ are written
${\vec{J}}=(J_t:t\in \aI_T)$ and in general we use 
the notions supplied in Section \ref{SecZMse}.
Similarly to Proposition \ref{prop5} we have:

\begin{proposition}
\label{prop6}
The measurable spaces $(\X, \sigma(X_{i,t}: (i,t)\!\in \! I\times \aI_T))$ and
$(\bP(I)^T, \bP(\bP(I)^T))$ are isomorphic by
$\Psi:\sigma(X_{i,t}: i\in I, t\in \aI_T)\to \bP(\bP(I)^T)$, where
\begin{equation}
\label{iso1p}
\forall\, {\vec {J}}\!\in \! \bP(I)^T:\;
\Psi\left(\bigcap_{t\in \aI_T}\bigcap\limits_{i\in J_t}X_{i,t} \setminus \!
\left(\bigcup_{{\vec {L}}: {\vec {L}}\supseteq {\vec {J}}, \, {\vec {L}}\neq {\vec {J}}}
\; \bigcap_{t\in \aI_T}\bigcap\limits_{i\in L_t}X_{i,t}\right)\!\right)=\{{\vec {J}}\};
\end{equation}
and on the other elements of the algebras we impose $\Psi$ preserves 
the disjoint unions, so $\Psi$ is an isomorphism of algebras.

\medskip

For every finite (respectively signed) measure $\nu$ defined on
$(\X,\sigma(X_{i,t}: i\!\in \! I, t\!\in \! \aI_T))$ the (respectively
signed) measure $\nu^*=\nu\circ \Psi^{-1}$ on $(\bP(I)^T,\bP(\bP(I)^T))$ is
\begin{equation*}
\forall\, {\vec {J}}\in \bP(I)^T:\;\; \nu^*({\vec {J}})=
\nu\left(\bigcap_{t\in \aI_T}\bigcap\limits_{i\in J_t}X_{i,t} \setminus
\left(\bigcup_{{\vec {L}}: {\vec {L}}\supseteq {\vec {J}}, \, {\vec {L}}\neq {\vec {J}}}
\; \bigcap_{t\in \aI_T}\bigcap\limits_{i\in L_t}X_{i,t}\right)\!\right)\,.
\end{equation*}
Under the isomorphism (\ref{iso1p}) we have:
\begin{equation*}
\Psi\left(\bigcap\limits_{t\in  \aI_T}\bigcap\limits_{i\in J_t}X_{i,t}\right)=
\{{\vec {K}}: {\vec {K}}\!\supseteq \!{\vec {J}}\}\,;
\end{equation*}
\begin{equation*}
\forall (i,t)\in I\times \aI_T:\; \Psi(X_{i,t})=\{{\vec{J}}: i\!\in \!J_t\}\,;
\end{equation*}
\begin{equation}
\label{re0cp}
\nu^*(\{{\vec{J}}\})=\sum_{{\vec{L}}: {\vec{L}}\supseteq {\vec{J}}}
(-1)^{\sum_{t\in \aI_T}(|L_t|-|J_t|)} \left(\sum_{{\vec{K}}: {\vec{K}}
\supseteq {\vec{L}}}\nu^*(\{{\vec{K}}\})\right).
\end{equation}
\end{proposition}

For any finite set $I$ the algebra $\sigma(X_i: i\in I)$ is
generated by the $2^{|I|}$ sets defined by (\ref{atom1}) 
(they could be less if some intersections are empty, but 
for this discussion assume this does not happen). Since the isomorphism 
of algebras must preserve the number of generating elements 
a Sylvester formula can be written with spaces having cardinality 
of the type $2^N$ (as $\bP(I)$) and this formula 
retrieves the M\"{o}bius matrix for sets (similarly for
product of sets). For partitions 
this way is useless because the cardinality of $\aP(N)$ does
not belong to the class of numbers $2^N$, except for some 
exceptional cases.

\subsubsection{Sylvester formula for partitions}
\label{SecSylv2}
We seek for a Sylvester formula for partitions that allows to 
retrieve the M\"{o}bius matrix for partitions (instead of for sets
as in the previous Section). 

\medskip

As noted, any measurable space $(\X,\sigma(X_i: \, i\in I))$ 
has $2^{|I|}$ generating elements defined by (\ref{atom1}). 
Then, no natural isomorphism of algebras 
can be established with a measurable space of the the type 
$(\aP(I'),\bP(\aP(I')))$ for some $I'$,
because the cardinality $|\aP(I')|$ is the Bell number $B_{|I'|}$ which 
in general is not of the type $2^N$.
So, we require to define an algebra by using other constructive mechanisms.
The basis for this construction is given by the following relation:
\begin{equation}
\label{atom2}
\hbox{For }  J\in \bP(I), \, \alpha\in \aP(I)\hbox{ we denote } J\vdash \alpha
\hbox{ if } \exists A\in \alpha \hbox{ such that } J\subseteq A\,.
\end{equation}

Let $(\X,\B(\X))$ be a measurable space 
and $(X_J: \,J\in \bP(I))$ be a family of sets  
indexed by $\bP(I)$. We define $\sigma^\aP(X_i: \,i\in \bP(I))$
as the $\sigma-$algebra of sets generated by the elements
\begin{equation}
\label{atom3}
\bigcap\limits_{J\vdash \alpha}X_J \setminus \!
\left(\bigcup_{\gamma: \gamma\succeq \alpha, \, \gamma\neq \alpha}
\; \bigcap\limits_{J\vdash \gamma}X_J\right),\;\, 
\alpha\!\in \! \aP(I).
\end{equation}
That is, the elements of $\sigma^\aP(X_i: i\in \bP(I))$ are 
all the finite unions of the sets defined in (\ref{atom3}).

\medskip

On the other hand note that every partition 
$\alpha\!\in \! \aP(I)$ satisfies
\begin{equation}
\label{atom4}
\{\alpha\}=\{\beta: \beta\succeq \alpha\}\setminus
\left(\bigcup_{\gamma: \gamma\succeq \alpha, \, \gamma\neq \alpha}
\{\beta: \beta\succeq \gamma\}\right).
\end{equation}

\begin{proposition}
\label{prop7}
The measurable spaces $(\X,\sigma^\aP(X_i: \,i\in \bP(I)))$ 
and $(\aP(I),\bP(\aP(I))$ are isomorphic by  
\begin{equation}
\label{iso1y}
\Psi:\sigma^\aP(X_i\!:\! i\!\in \!\bP(I))\to \bP(\aP(I))\!:\,
\forall \alpha\!\in \! \aP(I),\;
\Psi\!\left(\! \bigcap\limits_{J\vdash \alpha}X_J \setminus \!
\left(\bigcup_{\gamma: \gamma\succeq \alpha, \, \gamma\neq \alpha}
\; \bigcap\limits_{J\vdash \gamma}X_J\right)\!\right)\!=\!\{\alpha\},
\end{equation}
and we impose it preserves the
disjoint unions, so $\Psi$ is an isomorphism of algebras.
For every finite (respectively signed) measure $\nu$ defined on
$(\X,\sigma^\aP(X_i: \,i\in \bP(I)))$ the finite (respectively
signed) measure $\nu^*=\nu\circ \Psi^{-1}$ on $(\aP,\bP(\aP(I)))$
is given by:
\begin{equation}
\label{iso2y}
\forall\, \alpha\in \aP(I):\;\; \nu^*(\{\alpha\})=
\nu\left(\bigcap\limits_{J\vdash \alpha}X_J \setminus
\left(\bigcup_{\gamma: \gamma\succeq \alpha, \, \gamma\neq \alpha}
\; \bigcap\limits_{J\vdash \gamma}X_J\right)\right).
\end{equation}
Moreover, under the isomorphism (\ref{iso1y}) we have
\begin{equation}
\label{re0ya}
\Psi\left(\bigcap\limits_{J\vdash \alpha}X_J\right)=\{\beta: \beta\succeq \alpha\}\,;
\end{equation}
\begin{equation}
\label{re0yb}
\forall J\in \bP(I)\,:\; \Psi(X_J)=\{\alpha: J\vdash \alpha\}\,.
\end{equation}
\end{proposition}

\begin{proof}
Let us prove (\ref{re0ya}). From the isomorphism (\ref{iso1y}) and by setting
$X^*_J=\Psi(X_J)$ we get 
$$
\bigcap\limits_{J\vdash \alpha}X^*_J=
\Psi\left(\bigcap\limits_{J\vdash \alpha}X_J\right)=\{\beta: \beta\succeq \alpha\}.
$$
Then,
$$
X^*_J=\bigcup\limits_{\alpha: J\vdash \alpha} \, 
\bigcap\limits_{K: K\vdash \alpha}X^*_K=
\bigcup\limits_{\alpha: J\vdash \alpha}\{\beta: \beta\succeq \alpha\}.
$$
Now use,
$$
\bigcap\limits_{J\vdash \alpha} \{\beta: \beta\succeq \alpha\}=
 \{\beta: \forall J,\, J\vdash \alpha\Rightarrow J\vdash \beta\}
$$
to get  (\ref{re0yb}). 
\end{proof}

Let us now give the Sylvester formula in this setting. We recall 
the M\"{o}bius function $\mu$ defined in (\ref{mu1}).

\begin{proposition}
\label{prop8}
Let $\nu$ be a finite measure or a finite signed measure on the measurable space
$(\X,\sigma^\aP(X_i: \,i\in \bP(I)))$. Then,
\begin{equation}
\label{sylvpa}
\nu\left(\bigcap\limits_{J\vdash \alpha}X_J \setminus
\left(\bigcup_{\beta: \beta\succeq \alpha, \, \beta\neq \alpha}
\; \bigcap\limits_{J\vdash \beta}X_J\right)\right)
=
\sum\limits_{\beta: \beta\succeq \alpha} \mu(\alpha,\beta)
\nu\left(\bigcap\limits_{J\vdash \beta}X_J\right).
\end{equation}
\end{proposition}

\begin{proof}
Let $\nu^*=\nu\circ \Psi^{-1}$ be given by (\ref{iso2y}). 
Since the finite measure or signed measure spaces  
$(\X,\sigma^\aP(X_i: \,i\in \bP(I)),\nu)$ and $(\aP,\bP(\aP(I)),\nu^*)$
are isomorphic we get that (\ref{sylvpa}) is equivalent to
\begin{equation*}
\nu^*\left(\bigcap\limits_{J\vdash \alpha}X^*_J \setminus
\left(\bigcup_{\beta: \beta\succeq \alpha, \, \beta\neq \alpha}
\; \bigcap\limits_{J\vdash \beta}X^*_J\right)\right)
=
\sum\limits_{\beta: \beta\succeq \alpha} \mu(\alpha,\beta)
\nu^* \left(\bigcap\limits_{J\vdash \beta}X_J\right).
\end{equation*}
So, it is equivalent to
\begin{equation}
\label{sylvpaII}
\nu^*(\{\alpha\})=
\sum\limits_{\beta: \beta\succeq \alpha} \mu(\alpha,\beta)
\left(\sum\limits_{\gamma: \gamma\succeq \beta}\nu^*(\{\gamma\})\right),
\end{equation}
which is exactly $\nu^*(\{\alpha\})=({Z}^{-1}(Z\nu^*))(\{\alpha\})$ when 
$\nu^*=(\nu^*(\{\alpha\}: \alpha\in {\bP(I)})$ is written as a column vector. 
Hence, the result is shown.
\end{proof}

\subsection{M\"{o}bius positive cones for sets and partitions}
\label{SecConessp}

Below we describe the {M\"{o}bius positive cones 
$\F_+(\I)$, $\F'_+(\I)$, $\F(\I)=\F_+(\I)-\F_+(\I)$
and $\F'(\I)=\F'_+(\I)-\F'_+(\I)$ by using Sylvester formulae 
for the class of subsets and the set of partitions.

\subsubsection{M\"{o}bius positive cones for sets}
\label{SecConesse}
\begin{proposition}
\label{prop9}
\noindent We have that $g\in \F_+(\bP(I))$ (respectively 
$g\in \F(\bP(I))$) if and only if there exists
a finite measure (respectively a finite signed measure) $\nu^g$ defined
on the measurable space $(\X,\sigma(X_i: i\in I))$ that satisfies
\begin{equation}
\label{l0}
\forall\;  J\in \bP(I): \quad g(J)=\nu^g(\bigcap\limits_{i\in J} X_i).
\end{equation}
In this case,
\begin{equation}
\label{de1}
{Z^{-1}} g(J)=
\nu^g\left(\bigcap\limits_{i\in J} X_i\setminus
\left(\bigcup_{L: L\supseteq J, \, L\neq J} \;\,
\bigcap\limits_{i\in L}X_i \right)\right).
\end{equation}
Note that $g(\emptyset)=\nu^g(\X)$ and $Z^{-1} g(\emptyset)=
\nu^g(\X\setminus \bigcup_{i\in I} X_i)$ because
$\bigcap\limits_{i\in \emptyset}X_i=\X$.

\medskip

Moreover, if $g\in \F_+(\bP(I))$ (respectively  $g\in \F(\bP(I))$) the 
finite (respectively signed) measure $\nu^{g*}=\nu^g\circ \Psi^{-1}$
defined on $(\bP(I), \bP(\bP(I)))$ satisfies
\begin{equation*}
\forall J\in \bP(I)\,:\quad \nu^{g*}(\{J\})=Z^{-1} g (J) \, \hbox{ and }
\, g(J)=\sum_{K: K\supseteq J} \nu^{g*}(\{K\}).
\end{equation*}
The function $g\to \nu^{g*}$ 
defined from $\F(\bP(I))$ into the space of finite signed measures on 
$(\bP(I), \bP(\bP(I)))$, 
is linear and sends $\F_+(\bP(I))$ into the space of finite measures 
on $(\bP(I), \bP(\bP(I)))$.
\end{proposition}

\begin{proof}
\noindent Assume there exists a finite measure $\nu^g$ defined on
$(\X,\sigma(X_i: \,i\in I))$ such that $g$ satisfies 
(\ref{l0}). If $\nu^g$ is a measure, the expression on 
the right hand side of (\ref{de1}) is nonnegative because it is 
the measure of some event. We use (\ref{syl1}) to state the equality
in (\ref{de1}). Then,  
$Z^{-1} g(J)\ge 0$ for all $J\in \bP(I)$, so $g\in \F_+(\bP(I))$.
If $\nu^g$ is a signed measure we find $g\in \F(\bP(I))=\F_+(\bP(I))-\F_+(\bP(I))$.

\medskip

Conversely, if $g\in \F_+(\bP(I))$ we have ${Z^{-1}} g\ge 0$, so
we can define a measure $\nu^{g*}$ on $\bP(I)$ by the nonnegative weights
$\nu^{g*}(\{J\})=Z^{-1} g(J)$ for $J\in \bP(I)$.
By using that $\Psi$ is an isomorphism and the equality
$$
\{J\}= \bigcap\limits_{i\in J}X^*_i \setminus
\left(\bigcup_{L: L\supseteq J, \, L\neq J}
\; \bigcap\limits_{i\in L}X^*_i\right)\,,
$$
we conclude that the measure $\nu^g=\nu^{g*}\circ \Psi$ 
satisfies (\ref{de1}). 
Also, from the shape of $Z$ we get that
$g(J)=Z({Z^{-1}} g)(J)=\sum_{K: K\supseteq J} \nu^{g*}(\{K\})$ 
for all $J\in \bP(I)$.
Then $g$ satisfies (\ref{l0}).
The linearity property $g\to \nu^{g*}$ is a consequence of the linearity
of $Z^{-1}$ and the final statement on positivity of this application 
follows straightforwardly.
\end{proof}

\begin{proposition}
\label{prop10}
\noindent  We have $g\in \F'_+(\bP(I))$ (respectively $g\in \F'(\bP(I)))$ 
if and only if there exists
a finite measure (respectively a finite signed measure) $\nu_g$ defined
on a measurable space $(\X,\sigma(X_i: i\in I))$ that satisfies
\begin{equation*}
\forall\;  J\in \bP(I): \quad g(J)=\nu_g(\bigcap\limits_{i\in J^c} X_i).
\end{equation*}
In this case,
\begin{equation*}
{Z^{-1}} g(J)=
\nu_g\left(\bigcap\limits_{i\in J^c} X_i \setminus
\left(\bigcup_{L: L\supseteq J^c, \, L\neq J^c} \;\,
\bigcap\limits_{i\in L}X_i \right)\right).
\end{equation*}
Note that $g(I)=\nu_g(\X)$ and $Z^{-1} g(I)=
\nu_g(\X\setminus \bigcup_{i\in I} X_i)$.

\medskip

\noindent For each $g\in \F'_+(\bP(I))$ (respectively  $g\in \F'(\bP(I))$) 
the finite (respectively signed) measure $\nu_g^*=\nu_g\circ \Psi^{-1}$
on $(\bP(I), \bP(\bP(I)))$ satisfies
\begin{equation*}
\forall J\in \bP(I)\,:\quad \nu_g^*(J)={Z^{-1}}' g (J^c)\, \hbox{ and } \,
g(J)=\sum_{K: K\supseteq J^c} \nu_g^*(\{K\}).
\end{equation*}
The function $g\to \nu_g^*$
defined from $\F'(\bP(I))$ into the space of finite signed measures on 
$(\bP(I), \bP(\bP(I)))$, is linear and sends $\F'_+(\bP(I))$ 
into the space of finite measures on $(\bP(I), \bP(\bP(I)))$.
\end{proposition}

\begin{proof}
Define ${\widehat g}(J)=g(J^c)$, $J\in \bP(I)$.
We have
$$
{Z^{-1}}' g(J)=\sum_{K^c: K^c\subseteq J} (-1)^{|J|-|K^c|} g(K^c)
=\sum_{K: K\supseteq J^c} (-1)^{|K|-|J^c|}{\widehat g}(K)=
Z^{-1} {\widehat g}(J^c).
$$
Hence the result is a straightforward consequence
of Proposition \ref{prop9} applied to ${\widehat g}$.
\end{proof}

\subsubsection{M\"{o}bius positive cones for partitions}
\label{SecConespa}
Consider the M\"{o}bius positive cones $\F_+(\aP(I))$, $\F'_+(\aP(I))$ 
and the spaces $\F(\aP(I))=\F_+(\aP(I))-\F_+(\aP(I))$,
$\F'(\aP(I))=\F'_+(\aP(I))-\F'_+(\aP(I))$.
We shall describe them as we did in Propositions \ref{prop9} and \ref{prop10}.
But we will only write the statement for the cones $\F_+(\aP(I))$ 
and $\F(\aP(I))$. A similar statement can be written for  $\F'_+(\aP(I))$ and
$\F'(\aP(I))$, analogously as we did in Proposition \ref{prop10}.  

\begin{proposition}
\label{prop11}
$g\in \F_+(\aP(I))$ (respectively $g\in \F(\aP(I))$) if and only if there exists
a finite (respectively signed) measure $\nu^g$ defined on a 
measurable space $(\X,\sigma^\aP(X_J: J\in \bP(I)))$ that satisfies
\begin{equation}
\label{l0y}
\forall\;  \alpha\in \aP(I): \quad g(\alpha)=
\nu^g(\bigcap\limits_{J: J\vdash \alpha} X_J).
\end{equation}
In this case,
\begin{equation}
\label{de1y}
{Z^{-1}} g(\alpha)=
\nu^g\left(\bigcap\limits_{J: J\vdash \alpha} X_J \setminus
\left(\bigcup_{\gamma: \gamma\succeq \alpha, \, \gamma\neq \alpha} \;\,
\bigcap\limits_{J: J\vdash \gamma}X_J \right)\right).
\end{equation}

For each $g\in \F_+(\aP(I))$ the finite (respectively signed) measure 
$\nu^{*g}=\nu^g\circ \psi^{-1}$ defined on $(\aP(I), \bP(\aP(I)))$ satisfies
\begin{equation*}
\forall \alpha\in \aP(I)\,:\quad \nu^{*g}(\{\alpha\})=Z^{-1} g(\alpha) \,
\hbox{ and } \,  g(\alpha)=
\sum_{\beta: \beta\succeq \alpha} \nu^{*g}(\{\beta\}).
\end{equation*}
The function $g\to \nu^{*g}$ defined from $\F(\aP(I))$ 
into the space of finite signed measures of $(\aP(I), \bP(\aP(I)))$,
is linear and sends $\F_+(\aP(I))$ into the space of finite measures on 
$(\aP(I), \bP(\aP(I)))$.
\end{proposition}

\begin{proof}
Assume there exists a finite (respectively signed) measure $\nu^g$ defined
on $(\X,\sigma^\aP(X_J: \,J\in \bP(I)))$
such that $g$ satisfies (\ref{l0y}). Relation (\ref{de1y}) is a
consequence of formula (\ref{sylvpa}).
Since this formula is equivalent to (\ref{sylvpaII}), when $\nu^g$ is a measure
we have $Z^{-1} g(\alpha)\ge 0$ for all $\alpha\in \aP(I)$ and so
$g\in \F_+(\aP(I))$. When $\nu^g$ is a signed measure we find $g\in \F(\aP(I))$.

\medskip

Now, let $g\in \F_+(\aP(I))$, so ${Z^{-1}} g\ge 0$. We take the construction
of Proposition \ref{prop8}. We define the measure $\nu^{g*}$ on $\aP(I)$ 
by the nonnegative weights $\nu^{g*}(\{\alpha\})=Z^{-1} g(\alpha)$ for
$\alpha\in \aP(I)$. By using
$\bigcap\limits_{J: J\vdash \alpha} X^*_J=\{\beta: \beta\succeq \alpha\}$
and
$$
\{\alpha\}= \bigcap\limits_{J: J\vdash \alpha}X^*_J \setminus
\left(\bigcup_{\gamma: \gamma\succeq \alpha, \, \gamma\neq \alpha}
\; \bigcap\limits_{J: J\vdash \gamma}X^*_J\right),
$$
we get that $\nu^{g}=\nu^{g*}\circ \Psi$ satisfies (\ref{de1y})
(where $\Psi$ is defined in (\ref{iso1y})). From the shape of $Z$ we find
$$
\forall\, \alpha\in \aP(I)\,:\;\;
g(\alpha)=Z({Z^{-1}} g)(\alpha)=
\sum_{\beta: \beta\succeq \alpha} \nu^{g*}(\{\beta\}).
$$
Then $g$ satisfies (\ref{l0y}).
The linearity property $g\to \nu^{g*}$ is a consequence of the linearity
of $Z^{-1}$ and the final statement on positivity of this application
follows straightforwardly.
\end{proof}

\section{Coarse-graining}
\label{Sec4}
\subsection{Conditions for coarse-graining}
\label{SecCG}
As assumed $\I$ is a finite set. In this paragraph we do not 
require that it is partially ordered. Let $\sim$ be an 
equivalence relation on $\I$ and denote by $\wI$ the 
set of equivalence classes and by 
$\wa=\{b\in \I: b\sim a\}\in \wI$ the equivalence class containing $a$.
As always the equivalence classes are used, either as elements 
of $\wI$ or as subsets of $\I$. At each occasion it will be clear 
from the context in which of the two meanings we will be using them.

\medskip

A function $f:\I\to \RR$ is compatible with $\sim$ if $a\sim b$ implies 
$f(a)=f(b)$. In this case  
$\wf:\wI\to \RR$, $\wa\to \wf(\wa)=f(a)$ is a well defined function.

\medskip

A matrix $H=(H(a,b): a,b\in \I)$ is said to be compatible  
with $\sim$ if for any function $f:\I\to \RR$
compatible with $\sim$ the function $Hf$ is also compatible with
$\sim$, that is $a_1\sim a_2$ implies $Hf(a_1)=Hf(a_2)$.
Since  the set of compatible functions is a linear space generated by 
the characteristic functions of the sets we get that
$H$ is compatible with $\sim$ if and only if it verifies
the following condition, 
$$
\forall \, a_1\sim a_2, \, \forall \, \wb\in \wI\,: \quad
H \1_{\wb}(a_1)=H \1_{\wb}(a_2)\,,
$$
being $\1_{\wb}$ the characteristic function of the set $\wb\subseteq \I$.
Thus, $H$ is compatible with $\sim$ if it satisfies
the conditions known as those of coarse-graining,
\begin{equation}
\label{cgcond}
\forall \, a_1\sim a_2, \, \forall \, \wb\in \wI\,: \quad
\sum_{c\in \wb} H(a_1,c)=\sum_{c\in \wb} H(a_2,c)\,.
\end{equation}
Note that $\1=\sum_{\wb\in \wI} \1_{\wb}$.
So, if $H$ is compatible with $\sim$ we must necessarily have
$\sum_{c\in \I} H(a_1,c)=\sum_{c\in \I} H(a_2,c)$
when $a_1\sim a_2$. Hence, we have proven:
\begin{lemma}
\label{ldefcgm}
Assume $H$ is compatible with $\sim$. Then, the coarse-graining matrix 
$\wH=(\wH(\wa,\wb): \wa, \,\wb \in \wI)$ given by
$$
\forall  \wa, \wb\in \wI\,: 
\quad \wH(\wa,\wb)=\sum_{c\in \wb} H(a,c) \,,
$$   
is well defined and for every $f:\I\to \RR$ compatible with $\sim$ 
it holds
$$
\forall  \wa\in \wI\,: \quad  \wH \wf (\wa)=Hf(a)  \,.
$$
\end{lemma}

Note that if $H_1$ and $H_2$ are two matrices indexed by $\I\times \I$ 
compatible  with $\sim$ then $H_1+H_2$ and $H_1 H_2$
are compatible with $\sim$. For the sum this is a 
consequence of property (\ref{cgcond}).
For the product of matrices this property is also straightforward: 
let $f:\I\to \RR$ be a function compatible with $\sim$, 
then $H_2 f$ is function compatible with $\sim$ and so 
$H_1 H_2 f$ is also compatible with $\sim$,
proving that $H_1 H_2$ is compatible with $\sim$.

\medskip

Now, we claim that if $H$ is nonsingular and $H$ and $H^{-1}$ 
are both compatible with $\sim$,
then $\wH$ is nonsingular and its inverse $\wH^{-1}$ satisfies
$\wH^{-1}=\widetilde{H^{-1}}$, that is
\begin{equation*}
\forall \, \wa, \, \wb\in \wI\,: \quad
\wH^{-1}(\wa,\wb)=\widetilde{H^{-1}}(\wa , \wb)=
\sum_{c\in \wb} H^{-1}(a,c)\,.
\end{equation*}
In fact since $H$ and $H^{-1}$ are compatible with $\sim$
we get that for all $f:\I\to \RR$ compatible with $\sim$,
$$
\forall  \wa\in \wI\,: \quad
{\widetilde{H^{-1}}}(\wH \wf)(\wa)={\widetilde{H^{-1}}}
{\widetilde{H f}}(\wa)= \widetilde{H^{-1}(H f)(a)}=\wa\,.
$$

Note that for all equivalence relation $\sim$ 
the unit vector $\1$ is compatible with $\sim$. 
In the following result we exploit this fact. We denote by
$\wL$ the unit vector with the dimension of $\wI$.

\begin{lemma}
\label{lem1}
If $P$ is compatible with $\sim$ then the 
coarse-graining matrix $\wP$ preserves positivity,
stochasticity and substochasticity, that is
\begin{equation}
\label{dual3x}
P\ge 0 \, \Rightarrow \, \wP\ge 0\,; \;\;
P \1=\1 \, \Rightarrow \, \wP \wL=\wL\,; \;\;
P \1 \le \1 \, \Rightarrow \, 
\wP \wL \le \wL \,.
\end{equation}
\end{lemma}

\begin{proof}
The positivity is straightforward from the definition of $\wP$.
On the other hand since $\1$ is compatible with $\sim$, 
from $\wP\wL(\wa)=P \1(a)$ we get the last two
relations in (\ref{dual3x}).
\end{proof}

\begin{theorem}
\label{thm1}
Assume the duality relation $Q'=H^{-1} P H$ is satisfied. Let $\sim$
be an equivalence relation on $\I$ such that
the matrices $H$, $ H^{-1}$ and $P$ are compatible with $\sim$.
Then, $\wQ=(\wQ(\wa,\wb): \wa,\wb\in \wI)$ given by
\begin{equation}
\label{cgc2}
\wQ(\wa,\wb)=\sum_{c\in \wa} Q(c,b) \,,
\end{equation}
is a well defined matrix. It satisfies
$Q\ge 0 \, \Rightarrow \, \wQ\ge 0$ and
the following duality relation holds,
\begin{equation}
\label{dual1x}
\wQ'={\widetilde{H}}^{-1} \wP{\widetilde{H}}\,.
\end{equation}

For every strictly positive vector  
$\wah: \wI \to \RR_+$ the following duality relation holds
\begin{equation}
\label{dual4x}
{\wQ}'_{\wah^{\!-\!1}\!,\wah}={\wH}_\wah^{-1} \wP \wH_\wah \,
\hbox{ where } \,
\wH_\wah=\wH D_\wah^{-1} \, \hbox{ and }\, \wQ_{\wah^{\!-\!1}\!,\wah}= 
D_\wah^{-1} \wQ D_\wah\,,
\end{equation}
and positivity is preserved: $Q\ge 0 \, \Rightarrow \, 
{\wQ}_{\wah^{-1}\!,\wah} \ge 0$.

\medskip

For $\wah: \wI \to \RR_+$ defined by
\begin{equation}
\label{dual2x}
\wah(\wa)=|\wa|=|\{c\in \I: c\sim a\}|\,,
\end{equation}
the duality (\ref{dual4x}) preserves 
stochasticity and substochasticity of $Q$,
\begin{equation*}
\left(\,Q \1=\1 \, \Rightarrow \, \wQ_{\wah^{\!-\!1}\!,\wah} \wL
=\wL\right) \; \hbox{ and } \,
\left(\, Q\1\le \1 \, \Rightarrow \, \wQ_{\wah^{\!-\!1}\!,\wah} \wL
\le \wL\right)
\end{equation*}
Hence, if the kernels $P$ and $Q$ are stochastic 
(respectively substochastic) then the kernels $\wP$ and 
$\wQ_{\wah^{\!-\!1}\!,\wah}$ are stochastic (respectively substochastic).
\end{theorem}

\begin{proof}
From the hypotheses we get that $H^{-1}PH$ is compatible with $\sim$. Hence
$Q'=H^{-1}PH$ is compatible with $\sim$, and so $Q$ satisfies
\begin{equation*}
\forall \, b_1\sim b_2, \, \forall \, \wa\in \wI\,: \quad
\sum_{c\in \wa} Q(c,b_1)=\sum_{c\in \wa} Q(c,b_2)\,.
\end{equation*}
Hence $\wQ$ given by (\ref{cgc2}) is well defined on $\wI$. 
Let us show (\ref{dual1x}).
We must prove $\wQ'(\wa,\wb)=
\left({\widetilde{H^{-1}}}\wP{\widetilde{H}}\right)(\wa,\wb)$
for all $ \wa,\wb\in \wI$.
This relation is implied by the equality
$\sum_{c\in \wb} Q(c,a)=\sum_{c\in \wb} \left(H^{-1}PH\right) (a,c)$
for all $a\in \I$, $\wb\in \wI$,
and this last relation is fulfilled because the duality relation
(\ref{dual1}) is $Q(c,a)=(H^{-1}PH) (a,c)$ for all $a,c\in \I$.

\medskip

From Lemma \ref{lem1} it follows that coarse-graining preserves 
positivity, stochasticity and substochasticity of $P$.
On the other hand, by definition, we have that $Q\ge 0$ implies 
$\wQ\ge 0$. 

\medskip

Let $\wah:\wI\to \RR$ be a non-vanishing vector.
From (\ref{dual3}) we have that duality relation (\ref{dual1x}) implies
duality relation (\ref{dual4x}) for any non-vanishing vector $\wah$.
So, for $\wah$ strictly positive we get the implication
$Q\ge 0 \, \Rightarrow \, {\wQ}_{\wah^{\!-\!1}\!,\wah}=D_\wah^{-1} \wQ D_\wah \ge 0$.

\medskip
 
Now we define $h:\I\to \RR$ by $h(a)=\wah(\wa)$. The
duality relation $Q'=H^{-1} P H$ implies (\ref{dual3}) which is
$$
Q'_{h^{\!-\!1}\!,h}=H_h^{-1} P H_h \,
\hbox{ where } \,
H_h=H D_h^{-1} \, \hbox{ and }\, Q_{h^{\!-\!1}\!,h}= D_h^{-1} Q D_h\,.
$$
On the other hand the diagonal matrices
$D_h$ and $D_h^{-1}$ preserve $\sim$ and their coarse-graining matrices
are ${\widetilde {D_h}}=D_\wah$ and ${\widetilde {D_h^{-1}}}=D_\wah^{-1}$.
Then
$$
\left({\wH}_\wah^{-1} \wP \wH_\wah\right) (\wa,\wb)=\sum_{c\in \wb}
\left(H_h^{-1} P H_h\right) (a,c)=\sum_{c\in \wb}Q'_{h^{\!-\!1}\!,h}(a,c)=
\sum_{c\in \wb} D_h Q' D_h^{-1}(a,c).
$$
By the same argument and by definition of $\wQ_{\wah^{\!-\!1}\!,\wah}$ 
and $\wQ$ we get
$$
{\wQ}'_{\wah^{\!-\!1}\!,\wah}(\wa,\wb)=D_\wah \wQ' D^{-1}_\wah(\wa,\wb)
=\sum_{c\in \wb} D_h Q' D_h^{-1}(a,c).
$$ 
So, duality relation (\ref{dual4x}) is satisfied:
${\wQ}'_{\wah^{\!-\!1}\!,\wah}={\wH}_\wah^{-1} \wP \wH_\wah$.
Hence, (\ref{dual4}) implies that $\wQ_{\wah^{\!-\!1}\!,\wah} \wL=\wL$ is satisfied 
if and only if $\wQ\wah=\wah$, so if and only if
$\wah$ is a right eigenvector of $\wQ$ with eigenvalue $1$. Let us check that
$\wah$ defined in (\ref{dual2x}) is such an eigenvector. 

\medskip

We have
$$
\wQ \wah(\wa)=
\sum_{\wb\in \wI} \wQ(\wa,\wb)\wah(\wb)=
\sum_{\wb\in \wI} \sum_{c\in \wa} Q(c,b)\wah(\wb) 
=\sum_{\wb\in \wI} |\wb|\left(\sum_{c\in \wa} Q(c,b)\right).
$$
Since $\sum\limits_{c\in \wa} Q(c,d)$ does not depend on $d\in \wb$ we get
$|\wb|\left(\sum\limits_{c\in \wa} Q(c,b)\right)=\sum\limits_{d \in \wb} 
\sum\limits_{c\in \wa}Q(c,d)$ and so
\begin{equation}
\label{inter}
\wQ \wah(\wa)=\sum_{\wb\in \wI} \sum_{d \in \wb} \sum_{c\in \wa} Q(c,d)=
\sum_{c\in \wa}\left(\sum_{d\in \I} Q(c,d)\right).
\end{equation}
So, if $Q$ is stochastic we obtain $\sum_{d\in I} Q(c,d)=1$ for all $c\in I$ 
and we deduce
$$
\wQ \wah(\wa)=\sum_{c\in \wa} 1=|\wa|=\wah(\wa).
$$
We have shown that stochasticity is preserved: 
$Q \1=\1 \, \Rightarrow \, \wQ_{\wah^{\!-\!1}\!,\wah}\wL=\wL$. 

\medskip

The proof that substochasticity is also preserved 
is entirely similar. In fact the above arguments show the equivalence
$(\wQ_{\wah^{\!-\!1}\!,\wah} \wL \le \wL)
\Leftrightarrow (\,\wQ\wah\le \wah)$.
Now, $Q$ substochastic means $\sum_{d\in \I} Q(c,d)\le 1$ for all $c\in \I$.
We replace it in (\ref{inter}) to obtain $\wQ \wah(\wa)\le \wah(\wa)$
for $\wah$ given by (\ref{dual2x}). Therefore, the result is shown.
\end{proof}

As it is clear from the above computations, in general
$Q\1=\1$ (respectively $Q\1\le \1$) 
does not imply $\wQ\wL=\wL$
(respectively $\wQ\wL\le \wL$).
But it does when the function $\wah(\wa)=|\wa|$ is constant, 
because in this case $\wQ_{\wah^{\!-\!1}\!,\wah}=\wQ$. 

\medskip

A precision is required on transpose matrices and coarse-graining.
When the transpose matrix $H'$ is compatible with $\sim$, the 
matrix $\widetilde {H'}$ denotes its coarse-graining matrix. So,
$$
\widetilde {H'}(\wa,\wb)=\sum_{c\in \wb} H'(a,c)=\sum_{c\in \wb} H(c,a)\,.
$$
If $H$ is also compatible with $\sim$ then ${\widetilde H}\,'$ is the 
transpose of the coarse-graining matrix ${\widetilde H}$. 
In general the matrices $\widetilde {H'}$ and ${\widetilde H}\,'$
are not equal. In fact
$$
{\widetilde H}\,'(\wa,\wb)={\widetilde H}(\wb,\wa)=\sum_{c\in \wa} H(b,c)\,.
$$
Therefore we must take care in the notations. 
Thus, when $\widetilde {H'}$ is nonsingular the matrix 
${\widetilde {H'}}\,{}^{-1}$ is its inverse and if ${\widetilde H}$ 
is nonsingular then ${{\widetilde H}'}\,{}^{-1}={{\widetilde H}^{-1}}\,'$ 
is the inverse of the matrix ${\widetilde H}\,'$. In general 
the matrices ${\widetilde {H'}}\,{}^{-1}$ and ${{\widetilde H}\,'}\,{}^{-1}$ 
are not equal. But as noted, when these inverses exist we have 
the equalities 
$$ {\widetilde 
{H'}}\,{}^{-1}={\widetilde{{H'}{}^{-1}}} \hbox{ and } 
{{\widetilde H}\,'}\,{}^{-1}={{\widetilde {H^{-1}}}\,'}\,. 
$$

\subsection{Coarse-graining product formula} 
Let $(\I_r,\preceq_r)$ be a partially 
ordered space with M\"{o}bius functions $\mu_r$, for $r=1,2$. 
Recall that the M\"{o}bius function for the product space 
$(\I_1 \times \I_2,\preceq_{1,2})$ is 
given by $\mu((a_1,a_2),(b_1,b_2))=\mu_1(a_1,b_1)\mu_2(a_2,b_2)$ when 
$a_1\preceq_1 b_1,\, a_2\preceq_2 b_2$ (see (\ref{mu2})). 
Let $Z_r$ be the zeta matrix associated to $(\I_r,\preceq_r)$ 
for $r=1,2$ and $Z_{1,2}$ be the zeta matrix associated to the 
product space $(\I_1 \times \I_2,\preceq_{1,2})$.

\medskip

Let $\sim_1$ and $\sim_2$ be two equivalence relations
on $\I_1$ and $\I_2$ respectively. Then the product relation
$\sim_{1,2}$ defined on $\I_1 \times \I_2$ by 
$(a_1,a_2)\sim_{1,2} (b_1,b_2)$ if
$a_1\sim_1 b_1$ and $a_2\sim_2 b_2$, is an equivalence relation.
From definition we get 
${\widetilde{(a_1,a_2)}}=\wa_1\times \wa_2$ for all $(a_1,a_2)\in \I_1\times \I_2$.

\begin{proposition}
\label{prop12}
If $Z_r$ is compatible with $\sim_r$ for $r=1,2$, then
$Z_{1,2}$ is compatible with the product equivalence relation $\sim_{1,2}$
and the coarse-graining matrix is given by
\begin{equation}
\label{pcg1}
\widetilde{Z_{1,2}}((\wa_1,\wa_2),(\wb_1,\wb_2))=
{\widetilde{Z_1}}(\wa_1,\wb_1)\cdot {\widetilde{Z_2}}(\wa_2,\wb_2).
\end{equation}
Also, if $Z_r^{-1}$ is compatible with $\sim_r$ for $r=1,2$, then
$Z_{1,2}^{-1}$ is compatible with $\sim_{1,2}$ and
\begin{equation}
\label{pcg2}
\widetilde{Z_{1,2}^{-1}}((\wa_1,\wa_2),(\wb_1,\wb_2))=
{\widetilde{Z_1^{-1}}}(\wa_1,\wb_1)
\cdot {\widetilde{Z_2^{-1}}}(\wa_2,\wb_2).
\end{equation}
Similar statements and formulae can be stated for the transpose zeta 
and M\"{o}bius matrices.
\end{proposition}

\begin{proof}
Assume $(a_1,a_2)\sim (a'_1,a'_2)$.
From the product formula (\ref{pro1}) we get
$$
\{(c_1,c_2)\in \wb_1\times\wb_2: (c_1,c_2)\preceq_{1,2} (a_1,a_2)\}
= \{c_1\in \wb_1: c_1\preceq_{1} b_1\}\times
\{c_2\in \wb_2: c_2\preceq_{2} b_2\}.
$$
Then,
\begin{equation}
\label{pcg3}
Z_{1,2} \1_{ \wb_1\times\wb_2}(a_1,a_2)=
Z_1\1_{\wb_1}(a_1)Z_2 \1_{\wb_2}(a_1).
\end{equation}
(Also see (\ref{pro2})).
From our hypothesis we have $a'_r\sim_r a_r$ implies
$Z_r\1_{\wb_r}(a_r)=Z_r \1_{\wb_r}(a'_r)$ for $r=1,2$. Therefore,
$$
(a'_1, a'_2)\sim_{1,2} (a_1, a_2)  \Rightarrow
Z_{1,2} \1_{ \wb_1\times\wb_2}(a_1,a_2)=
Z_{1,2} \1_{ \wb_1\times\wb_2}(a'_1,a'_2).
$$
We have proven that the zeta matrix $Z_{1,2}$ is compatible with $\sim_{1,2}$.
Relation (\ref{pcg3}) gives (\ref{pcg1}).

\medskip

Now assume the M\"{o}bius matrices $Z_r^{-1}$ is compatible with $\sim_r$
so $Z_r^{-1}\1_{\wb_r}(a_r)=Z_r^{-1} \1_{\wb_r}(a'_r)$ in the
above setting, for $r=1,2$.
Also from the product formulae (\ref{pro1}) and (\ref{pro2}) we obtain
\begin{equation}
\label{pcg4}
Z_{1,2}^{-1} \1_{ \wb_1\times\wb_2}(a_1,a_2)=
Z_1^{-1}\1_{\wb_1}(a_1)Z_2^{-1} \1_{\wb_2}(a_1).
\end{equation}
Then,
$$
(a'_1, a'_2)\sim_{1,2} (a_1, a_2) \, \Rightarrow \,
Z_{1,2}^{-1} \1_{ \wb_1\times\wb_2}(a_1,a_2)=
Z_{1,2}^{-1} \1_{ \wb_1\times\wb_2}(a'_1,a'_2),
$$
proving that the  M\"{o}bius matrix $Z_{1,2}^{-1}$
is compatible with $\sim_{1,2}$. Equality (\ref{pcg4}) gives (\ref{pcg2}).
\end{proof}

\subsection{Coarse-Graining on zeta and M\"{o}bius matrices on sets and partitions}
\label{SecCGse2}
$(\I,\preceq)$ be a partially
ordered space with M\"{o}bius function $\mu$.
Let $\sim$ be an equivalence relation on $\I$. By definition, the zeta matrix
$Z$ is compatible with $\sim$ if and only if we have
\begin{equation}
\label{pcg5}
a_1\sim a_2 \Rightarrow \, \left(\forall \wb\in \wI: |\{c\in \wb: a_1\preceq c\}|=
 |\{c\in \wb: a_2\preceq c\}| \right).
\end{equation}
Similarly, $Z'$ is compatible with $\sim$ if and only if
\begin{equation}
\label{pcg6}
a_1\sim a_2 \Rightarrow \, \left(\forall \wb\in \wI: |\{c\in \wb: c\preceq a_1\}|=
 |\{c\in \wb: c\preceq a_2\}|\right),
\end{equation}
When the previous conditions hold we get
\begin{equation}
\label{pcg11}
{\widetilde Z}(\wa,\wb)=|\{c\in \wb: a\preceq c\}|\,,\quad 
{\widetilde {Z'}}(\wa,\wb)=|\{c\in \wb: c\preceq a\}|\,.
\end{equation}

Hence a sufficient condition for having zeta and M\"{o}bius 
compatibility with $\sim$ is the following one.

\begin{proposition}
\label{prop15}
Assume for all couple $a_1,a_2\in \I$ with $a_1\sim a_2$
there exists a bijection $\pi:\I\to \I$ such that: 
\begin{equation}
\label{pcg7}
\pi(a_1)=a_2\,;
\end{equation}
\begin{equation}
\label{pcg8}
c\preceq d\, \Leftrightarrow \, \pi(c)\preceq \pi(d)
\quad (\hbox{that is } \pi \hbox{ is an automorphism of } (\I,\preceq));
\end{equation}
\begin{equation}
\label{pcg9}
\forall \, \wb\in \wI:\;\, 
\pi(\wb)=\wb \hbox{ and } \pi:\wb\to \wb \hbox { is a bijection }.
\end{equation}
Then, $Z$, $Z'$, $Z^{-1}$ and ${Z'}{}^{-1}$ are compatible with $\sim$. 
\end{proposition}

\begin{proof}
The conditions imply
$$
\forall \wb\in \wI:\;\, 
\pi(\{c\!\in \!\wb: a_1\!\preceq \!c\})\!=\!\{c\!\in \!\wb: a_2\!\preceq \!c\},\;\,
\pi(\{c\!\in \!\wb: c\!\preceq \!a_1\})\!=\!\{c\!\in \!\wb: c\!\preceq \!a_2\}.
$$ 
Then, (\ref{pcg5}) and (\ref{pcg6}) are satisfied, so $Z$ and $Z'$ are
compatible with $\sim$. Since property (\ref{pcg8}) ensures
that $\pi$ is an isomorphism of $(\I,\preceq)$ into itself, then 
the M\"{o}bius function satisfies $\mu(c,d)=\mu(\pi(c),\pi(d))$ for all $c,d\in \I$. 
Hence, 
\begin{eqnarray*}
\sum_{c\in \wb}Z^{-1}(a_1,c)&=&\sum_{c\in \wb, a_1\preceq c} \mu(a_1,c)=
\sum_{\pi(c)\in \wb, \pi(a_1)\preceq \pi(c)} \mu(\pi(a_1),\pi(c))\\
&=&
\sum_{c\in \wb, a_2\preceq c} \mu(a_2,c).
\end{eqnarray*}
Similarly for ${Z'}^{-1}$. Then, the result is shown.
\end{proof}

\begin{remark}
\label{rem3}
Assume that the following property holds for all $a'\sim a$
and $b'\sim b$:
$$
\mu(a,b)=\mu(a',b') \hbox{ and } \left( a\preceq b
\,\Rightarrow a'\preceq b'\right)\,. 
$$
Then, $\; \wa {\widetilde{\preceq}} \wb \Leftrightarrow a\preceq b$
is a well defined order relation in $\wI$. Moreover, 
$\mu(\wa, \wb)=|\wb| \mu(a,b)$ is the  M\"{o}bius function
for $(\wI,{\widetilde{\preceq}})$. We have
${\widetilde Z}(\wa,\wb)=\1_{\wa{\widetilde{\preceq}} \wb}$ and 
${\widetilde Z}^{-1}(\wa,\wb)=\1_{\wa{\widetilde{\preceq}} \wb}
\mu(\wa, \wb)$. When the above properties are satisfied, they also hold
for the product equivalence relation and the product order.
\end{remark}

In the sequel, $I$ is a finite set and $N=|I|$ denotes 
its cardinality, so whenever needed we can assume $I=\aI_N$.

\subsubsection{Coarse-Graining on zeta and M\"{o}bius matrices on sets
and product of sets}
\label{SecCGse}
On $\I=\bP(I)$ consider the equivalence relation $\sim$ given by 
$J\sim K$ if $|J|=|K|$. In this case the set of equivalence classes
admits the following identification $\widetilde {\bP(I)}=\aI_N^0$ where  
$\aI_N^0=\{0,..,N\}$. 

\begin{proposition}
\label{prop13}
The matrices $Z$,  $Z^{-1}$ $Z'$ and ${Z'}{}^{-1}$
are all compatible with $\sim$. For $j,k\in \aI_N^0$ the $(j,k)$-entry
of the coarse-graining matrices are:
\begin{eqnarray}
\nonumber
&{}&{\widetilde Z}(j,k)=\binom{N-j}{k-j} \1_{j\le k}\,; \;\;
{\widetilde {Z^{-1}}}(j,k)=\binom{N-j}{k-j} (-1)^{k-j}\1_{j\le k}\,;\\
\label{coarset}
&{}& {\widetilde {Z'}}(j,k)=\binom{j}{k} \1_{k\le j}\, ; \;\;
{\widetilde {{Z'}{}^{-1}}}(j,k)=\binom{j}{k} (-1)^{j-k}\1_{k\le j}\,.
\end{eqnarray}
\end{proposition}

\begin{proof}
Let us check that the hypotheses of Proposition \ref{prop15} are satisfied.
Let $J,K\in \bP(I)$ be such that $|J|=|K|$. Let ${\widehat \pi}:I\to I$ 
be a bijection satisfying ${\widehat \pi}(J)=K$. 
Since $\mu(L,M)=(-1)^{|M|-|L|}$ when $L\subseteq M$, 
it is easy to see that $\pi: \bP(I)\to  \bP(I)$ defined by $\pi(L)=M$ 
(as elements) if and only if ${\widehat \pi}(L)=M$ (as sets), is
a bijection satisfying the hypotheses of Proposition \ref{prop15}. 
Then $Z$, $Z^{-1}$, $Z'$ and ${Z'}{}^{-1}$ are compatible with $\sim$.

\medskip

Let $j,k\in \aI_N^0$ and $J\in \bP(I)$ with $j=|J|$. When $k\ge j$ we have
$|\{L\in \bP(I): J\subseteq L, |L|=k\}|=\binom{N-j}{k-j}$. Also 
$\mu(J,L)=(-1)^{k-j}$ for any $L\supseteq J$ with $|L|=k$. This gives the first 
two equalities in (\ref{coarset}). On the other hand if $k\le j$
then $|\{L\in \bP(I): L\subseteq J, |L|=k\}|=\binom{j}{k}$ 
and $\mu(L,J)=(-1)^{j-k}$ for any $L\subseteq J$ with $|L|=k$. This gives 
the last two equalities in (\ref{coarset}). This finishes the proof.
\end{proof}

Let us consider the product space $\bP(I)^{T}$ endowed with the product 
order noted by
$\subseteq$. By he isomorphism (\ref{prx3}) all the relations and formulae
obtained for the class of sets continue to hold for the class 
of product of sets. 
Nevertheless, let us give explicitly the coarse-graining 
relations. On the class of product of sets we consider the
equivalence relation ${\vec J}\sim {\vec K}$ if $|J_t|=|K_t|$ for
all $t\in \aI_T$. 
Recall that the M\"{o}bius function of $(\bP(I)^T, \subseteq)$ is
$\mu({\vec J}, {\vec K})=
(-1)^{\sum_{t\in \aI_T}(|K_t|-|J_t|)}\1_{{\vec J}\subseteq {\vec K}}$.
From Propositions \ref{prop13} and \ref{prop12} we get that
the zeta matrix $Z$ and the M\"{o}bius matrix satisfy the coarse-graining
relations with respect to $\sim$. 

\medskip

The set of 
equivalence classes $\widetilde{\bP(I)^{T}}$ 
is naturally identified with $(I^0_N)^T$ which 
is endowed with the product partial order $\le$.
The elements of $(I^0_N)^T$ are written 
${\vec j}=(j_t: t\in \aI_T)$ and so,
${\vec j}\le {\vec k}$ when $j_t\le k_t \;\forall t\in \aI_T$. If
${\vec j}\le {\vec k}$ we denote
$$
\binom{\vec k}{\vec j}=\prod_{t\in \aI_T} \binom{k_t}{j_t}.
$$
With this notation the coarse-graining matrices are
\begin{eqnarray}
\nonumber
&{}&
{\widetilde Z}({\vec j},{\vec k})=\left(\prod\limits_{t\in \aI_T}
\binom{N-j_t}{k_t-j_t}\right)\1_{{\vec j}\le {\vec k}}\,;\\
\nonumber
&{}&
{\widetilde {Z^{-1}}}({\vec j},{\vec k})=
\left(\prod\limits_{t\in \aI_T} \binom{N-j_t}{k_t-j_t} \right)
(-1)^{\sum_{t\in I}(k_t-j_t)}\1_{{\vec j}\le {\vec k}}\,;\\
\nonumber
&{}& {\widetilde {Z'}}({\vec j},{\vec k})=
\binom{\vec j}{\vec k} \1_{{\vec k}\le {\vec j}}\,;\\
\nonumber
&{}&
{\widetilde {{Z'}{}^{-1}}}({\vec j},{\vec k})=
\binom{\vec j}{\vec k}
(-1)^{\sum_{t\in \aI_T}(j_t-k_t)}\1_{{\vec k}\le {\vec j}}\,.
\end{eqnarray}

We note that for the classes of sets and product of sets 
the conditions in Remark \ref{rem3} are satisfied.

\subsubsection{Coarse-Graining on zeta and M\"{o}bius matrices on partitions}
\label{SecCGpa}
Recall we can assume $I=\aI_N$. 
Let us define the decompositions of $N$ in an additive way. 
We set 
$$
\E_N=\{\eta:=\{e_s: s\in \aI_T\}: \, T\ge 1, \; e_s\ge 1 \, \forall s\in \aI_T, \,
\sum_{s\in \aI_T} e_s=N\}\,.
$$
Note that every $\eta\in \E_N$ is a multiset with elements $e_s\in \aI_N$ 
and with at most $T$ repetitions. The specificity is that 
the sum of the elements of $\eta\in \E_N$ is $N$.

\medskip
 
Let $[\eta]=T$ be the number of elements (including repetitions) 
of the multiset $\eta$. Let $\kappa=\{k_r:r\in \aI_R\}$ be another element 
in $\E_N$, we put
\begin{equation*}
\eta\wpreceq \kappa \, \Leftrightarrow \, T\ge R
\hbox{ and } \exists\, \theta: \aI_T\to \aI_R
\hbox{ onto such that} \!\!\!
\sum_{s\in \aI_T: \theta(s)=l} \!\!\!\!\!\!\! e_s=k_r  \; \forall r\!\in \! \aI_R.
\end{equation*}

For every partition $\alpha=\{A_t: t\in \aI_{[\alpha]}\}\in \aP(I)$ we denote by
$<\!\alpha\!>=\{|A_t|: t\in \aI_{[\alpha]}\}$ the multiset of the 
cardinal numbers of its atoms and call it the {\it skeleton} 
of the partition. We have $<\!\alpha\!>\in \E_N$
and $[<\!\alpha\!>]=[\alpha]$.

\medskip

On $\aP(I)$ we denote by $\alpha\sim \beta$ the equivalence relation
$<\!\alpha\!>=<\!\beta\!>$.

\medskip

Let us compute the number of partitions in $\aP(I)$ that has a certain skeleton.
For $\eta=\{e_s:s\in \aI_T\}\in \E_N$ define the equivalence relation
${\widehat{=}}_\eta$ on $\aI_T$ by $s_1 {\widehat{=}}_\eta\, s_2$ if $e_{s_1}=e_{s_2}$.
Let ${\widehat s}$ be the equivalence class of $s$ for the relation 
${\widehat{=}}_\eta$, so $|\widehat s|$ is the number of its elements. 
Denote by ${\wII}^\eta$ the set of equivalent classes. Define
$$
{\#}(\eta)  :=\binom{N}{\eta} \left(\prod_{{\tilde s}\in {\wII}^\eta}
|\widehat s|!\right)^{-1}\, \hbox{ with }
\binom{N}{\eta}:=\frac{N!}{\prod_{s\in \aI_T} e_s!}.
$$
We have that ${\#}(\eta) =|\{\alpha\in \aP(I): <\!\alpha\!>=\eta\}|$
is the number of different elements of $\aP(I)$ whose skeleton
is $\eta$, see equality $(1)$ in \cite{BG}. We recall that 
for a partition $\alpha$ and an atom $C\in \gamma$ of
a coarser partition $\gamma$, we denoted by
$\ell_C^\alpha$ the number of atoms of $\alpha$ contained in $C$,

\begin{proposition}
\label{prop14}
The matrices $Z$,  $Z^{-1}$, $Z'$ and ${Z^{-1}}'$
are all compatible with $\sim$ and the coarse-graining
matrices $\wZ=(\wZ(\eta,\kappa): \eta,\kappa \in \E_N)$
and $\wZ^{-1}=(\wZ^{-1}(\eta,\kappa): \eta,\kappa \in \E_N)$ satisfy:
\begin{eqnarray}
\nonumber
&{}&\wZ(\eta,\kappa)=|\{\gamma: <\!\gamma\!>=<\!\delta\!>, \alpha\preceq \gamma \}|
\1_{\eta\wpreceq \kappa}\;\,
\hbox{ for } <\!\alpha\!>=\eta, <\!\delta\!>=\kappa;\\
\label{coarsety}
&{}&\wZ^{-1}(\eta,\kappa)= \left(\!\sum_{\gamma: <\!\gamma\!>=<\!\delta\!>, 
\gamma\preceq \alpha} \!\!\!\!\!\!\!\!\!\!\!\!
(-1)^{[\alpha]+[\gamma]}\prod_{C\in \gamma} (\ell_C^\alpha\!-\!1)!\!\right)\!
\1_{\eta\wpreceq \kappa}\;\, \hbox{ for } 
<\!\alpha\!>\!=\!\eta, <\!\delta\!>\!=\!\kappa.
\end{eqnarray}
\end{proposition}

\begin{proof}
Let $\alpha, \beta$ be a pair of equivalent partitions in $\aP(I)$, so  
$ \alpha\sim \beta$. We will construct 
a permutation $\pi:\aP(I)\to \aP(I)$ that satisfies the properties 
(\ref{pcg7}), (\ref{pcg8}) and (\ref{pcg9}) of
Proposition \ref{prop15}, then the result will follow.

\medskip

We denote $T:=[\alpha]=[\beta]$. Let us fix an order to 
the atoms of $\alpha$, we denote by $\alpha^o=(A_t: t\in \aI_T)$ the 
ordered sequence. Since $<\!\alpha\!>=<\!\beta\!>$ we can fix an order 
$\beta^o=(B_t: t\in \aI_T)$ of the atoms of $\beta$ in such
a way that $|A_t|=|B_t|$ for $m\in \aI_T$.
We fix two permutations $\varphi_\alpha:I\to I$ and $\varphi_\beta:I\to I$
that satisfy
$$
\forall t\in \aI_T:\;\;
\varphi_\alpha(t)\in A_t \, \Leftrightarrow \, \varphi_\beta(t)\in B_t.
$$
Note that $\varphi=\varphi_\beta\circ \varphi_\alpha^{-1}$ is also
a permutation of $I$. We extend this permutation to the class
of partitions, we define $\pi: \aP(I)\to \aP(I)$ by
$$
\gamma=\{C_t: t\in \aI_T\}\to \pi(\gamma)=\{D_t: t\in \aI_T\}
$$
where the partition $\pi(\gamma)$ is given by the equivalence relation
$$
i{\equiv}_{\pi(\gamma)} j \, \Leftrightarrow \,
\pi^{-1}(i)\, {\equiv}_{\gamma} \, \pi^{-1}(j)\,.
$$
Since $\pi$ is defined by a pointwise permutation $\varphi$ in $I$,
it follows straightforwardly
that $\pi$ satisfies (\ref{pcg8}).
Also note that $\pi(\alpha)=\beta$, so  (\ref{pcg7}) holds. 
It is also clear from the definition 
of $\pi$ that it preserves the skeletons, that is 
$<\!\gamma\!>=<\!\pi(\gamma)\!>$. Then (\ref{pcg9}) is satisfied.

\medskip

Hence, from Proposition \ref{prop15} we get that $Z$, 
$Z^{-1}$, $Z'$ and ${Z^{-1}}'$ are compatible with $\sim$. 
The expression for $\wZ$ is the first equality in (\ref{pcg11}).
On the other hand,
\begin{eqnarray*}
\wZ^{-1}(<\!\alpha\!>,<\!\delta\!>)&=&
\sum_{\gamma: <\!\gamma\!>=<\!\delta\!>} \!\! Z^{-1}(\alpha, \gamma)= \!\!\!\!
\sum_{\gamma: <\!\gamma\!>=<\!\delta\!>, \gamma\preceq \alpha}\!\!\!\!\!\!\! 
\mu(\alpha,\gamma)\\
&=& \!\!\sum_{\gamma: <\!\gamma\!>=<\!\delta\!>, 
\gamma\preceq \alpha}\!\!\!\!\!\!\!
(-1)^{[\alpha]+[\gamma]}\prod_{C\in \gamma} (\ell_C^\alpha-1)!\;\,.
\end{eqnarray*}
Hence the equalities in (\ref{coarsety}) are satisfied.
Similar expressions can be found for  ${\widetilde{Z'}}$ and 
${\widetilde{Z'}}^{-1}$.
\end{proof}

\section{Examples}
\label{Sec5}
We will revisit the Cannings haploid and multi-allelic discrete population 
model with constant population size. The Cannings haploid discrete 
population model with constant population size \cite{Cn1, Cn2}
was introduced as a model encompassing the models of Wright-Fisher 
\cite{WF}, Moran \cite{Mo}, Kimura \cite{Ki} and Karlin and 
McGregor \cite{KM}. The multi-allelic model was introduced and 
studied in Gladstien and M\"ohle in \cite{Gl, MM2}. In \cite{MM1,MM2}
an ancestor type process was associated to the haploid 
and the multi-allelic models, and their duality was stated. 
We will provide a set version of these models and prove 
they are in duality via a transpose zeta matrix. The 
coarse-graining of the set model gives the Cannings model 
and the zeta transposed duality becomes an hypergeometric duality.

\subsection{Haploid Cannings model}
\label{SecSubs} 
The Canning haploid discrete population
model with constant population size (\cite{Cn1, Cn2}) 
was studied in a duality perspective in 
\cite{MM1}. There it was introduced an ancestor type
model which was proven to be in duality with the former 
one via an  hypergeometric matrix.

\medskip

Here, we construct an evolution model on the class of subsets of a 
fixed finite set whose coarse-graining is 
the Cannings haploid model. We also
construct an ancestor type model on the family of sets which is in
transpose zeta duality with the former one. The coarse-graining 
version of these kernels are the Cannings model and its
ancestor type model, and the transpose zeta matrix becomes 
the hypergeometric matrix.

\medskip

Let $I$ be a finite set, denote by $\lP(I)$ the class of indexed partitions
of $I$ defined by: $(J_i:i\in I)\in \lP(I)$ if
$$
\forall i\in I\; J_i\in \bP(I)\,,  \;\, 
\forall \, i\neq j \;\;  J_i\cap J_j=\emptyset \,, \;\;
\bigcup_{i\in I} J_i=I.
$$
Let $(\Omega, \BB, \PP)$ be a probability space and $\nu:\Omega\to \lP(I)$,
$\omega\to \nu(\omega)$ be a random element. Consider 
a collection of independent equally distributed 
random elements $(\nu^n: n\in \ZZ)$ with the law of $\nu$. The
elements indexed by nonnegative integers will serve to construct 
the haploid forward process and the elements with negative 
indexes will be at the basis of the definition of the backward 
process.

\medskip

Now we select a fixed allele and consider the set of individuals having this
allele. In time $n$ this set is called $X_n$ (so $I\setminus X_n$ is the
set of individuals having the another allele).
The evolution of the process $(X_n: n\in \NN)$ with values in $\bP(I)$ 
is given by
$$
X_{n+1}=\bigcup_{i\in X_n} \nu^{n+1}_i.
$$ 
The process $(X_n: n\in \NN)$ is a Markov chain with  
stochastic transition matrix $P=(P(J,K): J,K\in \bP(I))$ given by
\begin{equation}
\label{tra1}
P(J,K)=\PP(X_{n+1}=K \, | \, X_n=J)=
\PP\left(\bigcup_{i\in J} \nu_i=K\right). 
\end{equation}
This chain is called the forward process.

\medskip

Now we consider the process $(Y_n: n\in \NN)$ with values in 
$\bP(I)$ and defined recursively by
$$
\left(\bigcup_{i\in Y_{n+1}} \nu^{-(n+1)}_i\supseteq Y_n\right) \hbox{ and }
\left(\forall L\subseteq Y_{n+1}, L\neq Y_{n+1}:\, 
\bigcup_{i\in L} \nu^{-(n+1)}_i\not\supseteq Y_n\right)\,.
$$
Let us see that $Y_n\in \bP(I)$ defines a uniquely $Y_{n+1}\in \bP(I)$.
The existence of $Y_{n+1}$ follows from 
$\bigcup_{i\in I} \nu^{-(n+1)}_i\supseteq Y_n$.
In fact, if for all proper subset $L$ of $I$ we have 
$\bigcup_{i\in L} \nu^{-(n+1)}_i\not\supseteq Y_n$, then $Y_{n+1}=I$.
If there exists some proper subset $L_0$ such that
$\bigcup_{i\in L_0} \nu^{-(n+1)}_i\supseteq Y_n$, then 
we apply the above argument to the proper subsets of $L_0$, 
and we continue up to the moment when we find a subset satisfying the 
requirements of $Y_{n+1}$.
The uniqueness is a consequence of the disjointedness: 
$L\cap L'=\emptyset$ implies  $(\bigcup_{i\in L} \nu^{-(n+1)}_i)
\cap (\bigcup_{i\in L'} \nu^{-(n+1)}_i)=\emptyset$.
By definition, $Y_{n+1}$ can be seen as the set of ancestors of $Y_n$. 

\medskip

We have that $(Y_n: n\in \NN)$ is a Markov chain with stochastic 
transition matrix $Q=(Q(J,K): J,K\in \bP(I))$ given by
\begin{eqnarray}
\nonumber
Q(J,K)&=& \PP(Y_{n+1}=K \, | \, Y_n=J)\\
\nonumber
&=& \PP\left((\bigcup_{i\in K} \nu_i\supseteq J) \hbox{ and }
(\forall L\subseteq K, L\neq K:\, 
\bigcup_{i\in L} \nu_i\not\supseteq J)\right).
\end{eqnarray}
(The fact that $Q$ is stochastic is a consequence of the fact that
$Y_n\in \bP(I)$ determines $Y_{n+1}\in \bP(I)$). Define 
$\X_i=\{\nu_i\subseteq J^c\}$. We have
$$
\bigcap_{i\in K^c}\X_i=\{\bigcup_{i\in K^c}
\nu_i\subseteq J^c\}.
$$
Since $\bigcup_{i\in I}\nu_i=I$ and the sets 
$(\nu_i: i\in I)$ are disjoint, we deduce
$$
\bigcap_{i\in K^c}\X_i=\{\bigcup_{i\in K}\nu_i\supseteq J\}.
$$
Hence
\begin{eqnarray*}
&{}& \bigcap_{i\in K^c}\X_i\setminus 
\left(\bigcup_{L: L\subseteq K, L\neq K}\left(\bigcap_{i\in L^c}\X_i\right)\right)\\
&{}&
=\{\bigcup_{i\in K} \nu_i\supseteq J\}
\setminus \left(\bigcup_{L: L\subseteq K, L\neq K}
\{\bigcup_{i\in L}\nu_i\supseteq J\}\right).
\end{eqnarray*}
By the Sylvester formula we get,
\begin{eqnarray*}
Q(J,K)&=&\sum_{L\subseteq K} (-1)^{|K|-|L|} 
\PP\left(\bigcup_{i\in L} \nu_i\supseteq J\right)\\
&=& \sum_{L\subseteq K} (-1)^{|K|-|L|}\left(\sum_{M: M\supseteq J}
\PP\left(\bigcup_{i\in L} \nu_i= M\right)\right)\\
&=& \sum_{L\subseteq K} (-1)^{|K|-|L|}\left(\sum_{M: M\supseteq J}
P(L,M)\right).
\end{eqnarray*}

We can check that equality (\ref{es1x}) is satisfied, then 
the kernel $Q$ is the $Z'-$dual (transpose zeta dual) 
of $P$, that is $Q'={Z'}^{-1} P Z'$ is satisfied where the
$Z'$ matrix is given by $Z'(J,K)=\1_{K\subseteq J}$. 

\medskip

Now assume the law of $\nu$ is invariant under permutation of
$I$, this means for all permutation $\pi=(\pi_i: i\in I)$ of $I$ we have
\begin{equation}
\label{per1}
\forall (J_i:i\!\in \!I)\in \lP(I):\;\,
\PP(\nu_i\!=\!J_i: i\!\in \!I)=\PP(\nu_{\pi(i)}\!=\!J_i: i\!\in \! I).
\end{equation}

As in Section \ref{SecCGse} let us take on $\bP(I)$ the equivalence 
relation given by the cardinality, $J\sim K$ if $|J|=|K|$.
Recall $\aI_N^0=\{0,..,N\}$ is identified with the
set of equivalence classes. 
Let us check that $P$ satisfies the coarse-graining conditions.
For $m\in \aI_N^0$ set 
$\Gamma^*(m)=\{L\subseteq I: |L|=m\}$. We must verify that,
\begin{equation}
\label{tra3}
|J|=|K|\, \Rightarrow \, \forall m\in \aI_N^0: \; \sum_{L\in \Gamma^*(m)} P(J,L)=
\sum_{L\in \Gamma^*(m)} P(K,L).
\end{equation}
Let $\pi$ be any permutation of $I$ such that $\pi(J)=K$.
We have that $\pi:\Gamma^*(m)\to \Gamma^*(m)$, $L\to \pi(L)$, is a bijection.
From (\ref{tra1}) and (\ref{per1}) we have 
$$
P(J,L)=\PP\left(\bigcup_{i\in J} \nu_i=L\right)=
\PP\left(\bigcup_{i\in J} \nu_{\pi(i)}=L\right)=
\PP\left(\bigcup_{i\in K} \nu_i=L\right)=P(K,L).
$$
Hence $\sum_{L\in \Gamma^*(m)}P(J,L)=\sum_{L\in \Gamma^*(m)}P(K,L)$,
and so (\ref{tra3}) is satisfied.

\medskip

The coarse-graining matrix $\wP=(\wP(i,j: i,j\in \aI_N^0)$ satisfies
\begin{equation*}
\hbox{For } |J|=i:\;\, 
\wP(i,j)=\sum_{L: |L|=j}P(J,L)
=\sum_{L: |L|=j}\PP\left(\bigcup_{i\in J} \nu_i=L\right).
\end{equation*}
Let us show $\wP$ is the transition matrix of the forward process for the
haploid model of Cannings (\cite{Cn1}). Let 
$$
{\widehat{\E}}_N=
\{{\vec e}=(e_1,..,e_N)\in (I^0_N)^N: \sum_{i=1}^N e_i=N \}. 
$$
Define the random element 
$|\nu|:\Omega\to {\widehat{\E}}_N$, $\omega\to |\nu(\omega)|$, that is
$|\nu(\omega)|_i=|\nu_i(\omega)|$
is the number of elements of the set $\nu_i(\omega)$. Note that 
$\sum_{i\in I}|\nu_i(\omega)|=N$ because $\nu(\omega)$ is an indexed partition.

\medskip

Since the law of $\nu$ is invariant by permutations, see (\ref{per1}),
we get that the law of $|\nu|$ is exchangeable, that is for all permutation
$\pi$ of $\aI_N$ it is satisfied
$$
\forall {\vec e}\in {\widehat{\E}}_N:\;\;
\PP(|\nu_{\pi(i)}|=e_i, i\in \aI_N)=\PP(|\nu_i|=e_i, i\in \aI_N).
$$
On the other hand we have
$$
\PP(\sum_{l=1}^i |\nu_l|=j)=\sum_{J: |J|=j}\PP(\bigcup_{l=1}^i \nu_l=J).
$$
Hence, the coarse-graining kernel $\wP$ satisfies
$$
\wP(i,j)=\PP(\sum_{l=1}^i |\nu|_l=j).
$$
Then $\wP$ is the kernel of the forward process of the haploid model of Cannings.
Denote $H=Z'$. Let us compute $\wH=\widetilde{Z'}$ in this coarse-graining setting.
Let $i,j\in \aI_N^0$, take $J$ be such that $|J|=i$, we have
$$
\wH(i,j)=\sum_{L: |L|=j} Z'(J,L)=\sum_{L: |L|=j} 
\1_{L\subseteq J}=|\{L: |L|=j, L\subseteq J\}|=\binom{i}{j}\1_{i\geq j}.
$$
In this case the function of (\ref{dual2x}) is
$\wah(j)=|\{L\subseteq I: |L|=j\}|=\binom{N}{j}$ for $j\in \aI_N^0$. Then 
$\wH_\wah=\wH D_\wah^{-1}$ satisfies
$$
\wH_\wah(i,j)=\frac{\binom{i}{j}}{\binom{N}{j}}\1_{i\geq j}.
$$
An easy computation gives,
$$
\wH_\wah^{-1}(i,j)=(-1)^{i-j}\binom{i}{j}\binom{N}{i}\1_{i\geq j} \,.
$$
Therefore, Theorem \ref{thm1} ensures that the 
matrix $\wQ_{\wah^{\!-\!1}\!,\wah}=
D_\wah^{-1} \wQ D_\wah$ is a stochastic matrix that satisfies
\begin{equation*}
{\wQ}'_{\wah^{\!-\!1}\!,\wah}={\wH}_\wah^{-1} \wP \wH_\wah\,.
\end{equation*}
So, it is the $\wH_\wah-$dual of $\wP$, see (\ref{dual4x}). 

\medskip

The matrix $\wH_\wah$, called the hypergeometric matrix, was firstly 
introduced in \cite{MM1} as a dual kernel between the forward 
process and the backward process of the haploid model of Cannings. 
As said, as a consequence of our results, the transition matrix of 
the backward process is given by $\wQ_{\wah^{\!-\!1}\!,\wah}= D_\wah^{-1} \wQ D_\wah$. 
In \cite{MM1} it is proven that this transition matrix also satisfies
$$
\wQ_{\wah^{\!-\!1}\!,\wah}(i,j)=\frac{\binom{N}{j}}{\binom{N}{i}}
\sum_{(l_1,..,l_j)\in (\aI_N^0)^j: \sum_{r=1}^j l_j=i}
\EE\left(\prod_{r=1}^j \binom{|\nu_i|}{l_i}\right).
$$

\subsection{Multi-allelic Cannings model}
Here we construct a multi-allelic model on the product class of 
subsets of a fixed finite set. We also
construct an ancestor type process on the family of sets which is in
transpose zeta duality with the former one.
We show that the coarse-graining of these models are the multi-allelic 
Cannings model as introduced in \cite{Gl, MM2} and the associated
ancestor type process defined in \cite{MM2}. The coarse-graining of the 
transpose zeta matrix becomes a generalized hypergeometric matrix.

\medskip

Let $I$ be a finite set and $T$ be the number of types. We 
assume $T\ge 2$. Consider two different classes of product of sets:
\begin{eqnarray*}
\lP^{(T)}(I) &=& \{ \vec{J}:=(J_t: t\!\in \! \aI_T): \, \forall \, t \; J_t\!\in \! \bP(I), 
\; t\!\neq \! t' \;\, J_t\cap J_{t'}\!= \!\emptyset, \; \bigcup_{t\in \aI_T} J_t\!= 
\!I \}\,;\\ \lS^{(T)}(I)&=& \{\vec{J}=(J_t: t\!\in \! \aI_T): \, 
\forall t \; J_t\in \bP(I), \; t\neq t' \;\; J_t\cap J_{t'}=\emptyset\}\,.
\end{eqnarray*}
That is, the elements $(J_t: t\!\in \! \aI_T)\!\in \! \lS^{(T)}(I)$ do not necessarily
cover $I$ (they satisfy $\bigcup_{t\in \aI_T} J_t\subseteq I$).
Note that $\lP^{(T)}(I)\subseteq \bP(I)^T$ and $\lS^{(T)}(I)\subseteq \bP(I)^T$.

\medskip

As before, $(\Omega, \BB, \PP)$ is a probability space and $\nu:\Omega\to \lP(I)$,
$\omega\to \nu(\omega)$ is a random element. 
Consider a collection of independent equally distributed
random elements $(\nu^n: n\in \ZZ)$ with the law of $\nu$. 
The elements indexed by a nonnegative $n$ will serve to construct
the forward process and the elements with negative
$n$ will serve to define the backward process.

\medskip

Let us define the process $(X_n: n\in \NN)$ with values in $\lP^{(T)}(I)$.
The $t-$coordinate of $X_n$ is noted by $(X_n)_t$. The process is given
by,
$$
\forall t\in \aI_T:\;\; (X_{n+1})_t=\bigcup_{i\in (X_n)_t} \nu^{n+1}_i.
$$
The process $(X_n: n\in \NN)$ is well defined in $\lP^{(T)}(I)$,
that is $X_0\in \lP^{(T)}(I)$ implies $X_n\in \lP^{(T)}(I)$ for
all $n\in \NN$, because $\nu$ takes values in $\lP(I)$.

\medskip
 
The process $(X_n: n\in \NN)$ is a Markov chain with stochastic transition 
matrix $P=(P(\vec{J},\vec{K}): \vec{J},\vec{K}\in \lP^{(T)}(I))$ given by
\begin{equation}
\label{tra11}
P(\vec{J},\vec{K})=\PP(X_{n+1}=\vec{K} \, | \, X_n=\vec{J})=
\PP\left(\bigcap_{t\in \aI_T}\left( \bigcup_{i\in J_t} \nu_i=K_t\right)\right).
\end{equation}
This chain is called the forward process.

\medskip

Now we define the backward process $(Y_n: n\in \NN)$ which will take 
values in $\lS^{(T)}(I)$. The $t-$coordinate of $Y_n$ will be 
denoted by $(Y_n)_t$. To define the process 
it is useful to use the product order on $\bP(I)^T$: 
${\vec{L}}\subseteq {\vec{M}}$
when $L_t\subseteq M_t$ for $t\in \aI_T$.
We define $Y_{n+1}$ from $Y_n$ by:
\begin{eqnarray*}
&{}& 
\left(\bigcap_{t\in \aI_T}\; \bigcup_{i\in (Y_{n+1})_t} 
\nu^{-(n+1)}_i\supseteq (Y_n)_t\right) \hbox{ and }\\
&{}&
\left(\forall \vec{L}\subseteq Y_{n+1}, \vec{L}\neq Y_{n+1}\,:\; 
\bigcup_{t\in \aI_T}
\left(\bigcup_{i\in L_t} \nu^{-(n+1)}_i\not\supseteq (Y_n)_t\right)\right).
\end{eqnarray*}
In this case it is not guaranteed that for all 
$Y_n\in \lS^{(T)}(I)$ there exists 
some $Y_{n+1}\in \lS^{(T)}(I)$ satisfying the above requirements.
But when it exists it is uniquely defined because of the disjointedness
property:
$L\cap L'=\emptyset$ implies  $(\bigcup_{i\in L} \nu^{-(n+1)}_i)
\cap (\bigcup_{i\in L'} \nu^{-(n+1)}_i)=\emptyset$.

\medskip
 
The random set $Y_{n+1}$ can be thought as the set of ancestors of $Y_n$. 
The process $(Y_n: n\in \NN)$ is a Markov chain that can lose mass. 
Its evolution is given by the (substochastic) transition matrix
$Q=\left(Q(\vec{J},\vec{K}): \vec{J},\vec{K}\in \lS^{(T)}(I)\right)$ 
defined by
\begin{eqnarray*}
&{}& Q(\vec{J},\vec{K})= \PP(Y_{n+1}=\vec{K} \, | \, Y_n=\vec{J})\\
&{}&= \PP\left(( \forall \, t\!\in \! \aI_T\;  
\bigcup_{i\in K_t} \nu_i\supseteq J_t ) 
\hbox{ and } (\forall  \vec{L}\subseteq Y_{n+1}, {\vec{L}}\neq {\vec{K}},\,
\exists t\!\in \! \aI_T\,: \bigcup_{i\in L_t} \nu_i\not\supseteq J_t)\right).
\end{eqnarray*}

Let us relate both kernels $Q$ and $P$. To this purpose 
it is convenient to define $A_{i,t}=\{\nu_i\subseteq J_t^c\}$.
As before, from $\bigcup_{i\in I}\nu_i=I$ and the disjointedness of 
the sets $(\nu_i: i\in I)$ we get
$$
\bigcap_{i\in K_t^c}A_{i,t}=\{\bigcup_{i\in K_t}\nu_i\supseteq J_t\}.
$$
Let us consider
\begin{equation}
\label{aux1}
\X_{\vec{K}}^{(\vec{J})}=\bigcap_{t\in \aI_T}\left(\bigcap_{i\in K_t^c}A_{i,t}\right)
=\bigcap_{t\in \aI_T}\left(\{\bigcup_{i\in K_t}\nu_i\supseteq J_t\}\right).
\end{equation}
Hence, relation (\ref{aux1}) and the product order allows to write,
\begin{equation*}
Q(\vec{J},\vec{K})=
\PP\left(\X_{\vec{K}}^{(\vec{J})}
\setminus \left(\bigcup_{ {\vec{L}}\subseteq {\vec{K}}, 
{\vec{L}}\neq {\vec{K}} }\X_{\vec{L}}^{(\vec{J})}\right)\right).
\end{equation*}
Note that
$$
\PP(\X_{\vec{K}}^{(\vec{J})})=\sum_{{\vec M}: {\vec M}\supseteq {\vec J}} 
P({\vec K},{\vec L}).
$$ 
By the Sylvester formula (\ref{re0cp}) for product of sets we get
\begin{equation*}
Q({\vec{J}},{\vec{K}})=\sum_{{\vec{L}}: {\vec{L}}\subseteq {\vec{K}}}
(-1)^{\sum_{t\in \aI_T}(|K_t|-|L_t|)} \left(\sum_{{\vec{M}}: {\vec{M}}
\supseteq {\vec{J}}}P({\vec{L}}, {\vec{M}})\right).
\end{equation*}
Hence, equality (\ref{es1x}) is satisfied, then $Q'={Z'}^{-1} P Z'$
holds with $Z'$ given by $Z'({\vec{J}},{\vec{K}})=
\1_{{\vec{K}}\subseteq {\vec{J}}}$. That is,
the kernel $Q$ is the $Z'-$dual (transpose zeta dual) of $P$.

\medskip

Now assume the law of $\nu$ is invariant under permutation of
$I$, so (\ref{per1}) is satisfied.
In $\bP(\I)^T$ we define $|{\vec{J}}|=(|J_t|: t\in \aI_T)$ and
we endow $\bP(\I)^T$ with the equivalence
relation ${\vec{J}}\sim {\vec{K}}$ if 
$|{\vec{J}}|=|{\vec{K}}|$.
Let us check that $P$ satisfies the coarse-graining conditions.
Fix $N=|I|$, let
$$
{\widehat{\E}}_N^{(T)}=\{{\vec e}=(e_1,..,e_T)\in (\aI_N^0)^T: 
\sum_{t\in \aI_T}e_t=N\}.
$$ 
For all ${\vec e}\in {\widehat{\E}}_N^{(T)}$ we define
$$
\Gamma^*({\vec e})=\{{\vec {L}}\in \lP^{(T)}(I): |{\vec {L}}|={\vec e}\}.
$$ 
We must verify that,
\begin{equation*}
|{\vec {J}}|=|{\vec {K}}|\, \Rightarrow \, \forall 
{\vec e}\in {\widehat{\E}}_N^{(T)}: \; 
\sum_{{\vec {L}}\in \Gamma^*(\vec e)} 
P({\vec {J}},{\vec {L}})=
\sum_{{\vec {L}}\in \Gamma^*(\vec e)} P({\vec {K}},{\vec {L}}).
\end{equation*}
Let $\pi$ be any permutation of $I$ such that $\pi(J_t)=K_t$ for all 
$t\in \aI_T$, this permutation exists because the elements of
$\vec J$ are disjoint sets, as well as those of $\vec K$.
We have that $\pi:\Gamma^*({\vec e})\to \Gamma^*({\vec e})$, 
${\vec {L}}\to \pi({\vec {L}})$, is a bijection.
From (\ref{tra11}) and (\ref{per1}) we have
\begin{eqnarray*}
P({\vec {J}},{\vec {L}})&=&
\PP\left(\bigcap_{t\in \aI_T}\left(\bigcup_{i\in J_t} \nu_i=L_t\right)\right)
=\PP\left(\bigcap_{t\in \aI_T}\left(\bigcup_{i\in J_t} 
\nu_{\pi(i)}=L_t\right)\right)\\
&=&\PP\left(\bigcap_{t\in \aI_T}\left(\bigcup_{i\in K_t} \nu_i=L_t\right)\right) 
=P({\vec {K}},{\vec {L}}).
\end{eqnarray*}
Hence $\sum_{{\vec {L}}\in \Gamma^*({\vec e})}
P({\vec {J}},{\vec {L}})=\sum_{{\vec {L}}\in \Gamma^*({\vec e})}
P({\vec {K}},{\vec {L}})$, and so (\ref{tra3}) is satisfied.

\medskip

The coarse-graining matrix $\wP=(\wP({\vec d},{\vec e}): 
{\vec d},{\vec e}\in {\widehat{\E}}_N^{(T)})$ is such that for all 
${\vec d},{\vec e}\in {\widehat{\E}}_N^{(T)}$ 
and every ${\vec{J}}$ that satisfies $|{\vec{J}}|\!=\!{\vec d}$,
\begin{equation}
\label{duP11}
\wP({\vec d},{\vec e})=\!\! \sum_{{\vec {L}}\in \Gamma^*({\vec e})}
\!\!P({\vec{J}},{\vec{L}})
=\sum_{{\vec {L}}\in \Gamma^*({\vec e})}\!\!
\PP\left(\bigcap_{t\in \aI_T}\left(\bigcup_{i\in J_t} \nu_i=L_t\right)\right).
\end{equation}
Let us show $\wP$ is the transition matrix of the forward process for the
multi-allelic model in \cite{Gl, MM2}. Recall the random element
$|\nu|:\Omega\to {\widehat{\E}}_N^{(T)}$, $\omega\to |\nu(\omega)|$, so
$|\nu(\omega)|_i=|\nu_i(\omega)|$.
As pointed out, since the law of $\nu$ is invariant under permutations,
the law of $|\nu|$ is exchangeable. From exchangeability 
and relation (\ref{duP11}), the coarse-graining matrix $\wP$ satisfies
for  all pair ${\vec d},{\vec e}\in {\widehat{\E}}_N^{(T)}$,
$$
\wP({\vec d},{\vec e})=
\PP\left(\bigcap_{t\in \aI_T}\left(\sum_{l=\Delta_{t-1}+1}^{\Delta_t}
 |\nu|_l=e_t\right)\right) \hbox{ where }
\Delta_t=\sum_{s=1}^t d_s \hbox{ for } t\in \aI_T \hbox{ and } \Delta_0=0.
$$
Then $\wP$ is the kernel of the forward process of the multi-allelic model 
in \cite{Gl, MM2}. let $H=Z'$.
Let us compute the dual matrix $\wH={\widetilde {Z'}}$ in 
this coarse-graining setting. Let ${\vec d},{\vec e}\in {\widehat{\E}}_N^{(T)}$ 
and ${\vec{J}}$ be such that $|{\vec{J}}|={\vec d}$, we have
\begin{eqnarray*}
\wH({\vec d},{\vec e})&=&\sum_{{\vec {L}}\in \Gamma^*(\vec e)} 
Z'({\vec{J}},{\vec{L}})=\sum_{{\vec {L}}\in \Gamma^*(\vec e)}
\1_{{\vec{L}}\subseteq {\vec{J}}}=
|{\vec {L}}\in \Gamma^*(\vec e), \, {\vec{L}}\subseteq {\vec{J}}\}|\\
&=&\prod_{t\in \aI_T}\binom{d_t}{e_t}\1_{{\vec d}\geq {\vec e}}
:=\binom{\vec d}{\vec e} \1_{{\vec d}\geq {\vec e}}\;.
\end{eqnarray*}
In this case the function $\wah$ of (\ref{dual2x}) is given by,
\begin{equation}
\label{mm23}
\forall\; {\vec e}\in {\widehat{\E}}_N^{(T)}:\;\;
\wah({\vec e})=|\Gamma^*({\vec e})|=
\frac{N!}{\prod_{t\in \aI_T} e_t!}=\binom{N}{{\vec e}}\,. 
\end{equation}
Then
$\wH_\wah=\wH D_\wah^{-1}$ satisfies
\begin{equation}
\label{mm2}
\wH_\wah({\vec d},{\vec e})=\frac{\binom{\vec d}{\vec e}}{\binom{N}{{\vec e}}}
\1_{{\vec d}\geq {\vec e}}\;.
\end{equation}
Therefore, Theorem \ref{thm1} ensures that the
matrix $\wQ_{\wah^{\!-\!1}\!,\wah}=D_\wah^{-1} \wQ D_\wah$ 
is substochastic (because $Q$ is) and satisfies
\begin{equation*}
{\wQ}'_{\wah^{\!-\!1}\!,\wah}={\wH}_\wah^{-1} \wP \wH_\wah\,,
\end{equation*}
that is, is the $\wH_\wah-$dual of $\wP$, see (\ref{dual4x}).

\medskip

We note that the coefficients of (\ref{mm2}) are exactly 
the same as those appearing in expression $(8)$ in \cite{MM2}.
Hence, by coarse-graining  
we have retrieved the result proven in \cite{MM2}, that a dual 
kernel between the forward process and the backward process 
of the multi-allelic model of Cannings is given by (\ref{mm2}).
In \cite{MM2} it is supplied several formulae for 
$\wQ_{\wah^{\!-\!1}\!,\wah}$, in particular see its Proposition $2$.

\section*{Acknowledgements} 
The authors acknowledge the partial support given by the CONICYT BASAL-CMM
project PFB $03$ and S. Mart\'{i}nez thanks the hospitality of
{Laboratoire de Physique Th\'{e}orique et Mod\'{e}lisation at the
Universit\'{e} de Cergy-Pontoise. The authors thanks an anonymous referee
for his/her comments that allow to improve the presentation of this work.

\end{document}